\documentclass[12pt]{elsarticle}
\usepackage[utf8]{inputenc}
\usepackage{amsmath}
\usepackage{amsfonts}
\usepackage{amssymb}
\usepackage{amsthm}
\usepackage{tikz-cd}
\usepackage{mathtools}
\usepackage{hyperref}
\numberwithin{equation}{section}
\usepackage{subfiles}
\usepackage{graphicx}
\usepackage{ytableau}
\usepackage{xcolor}
\usepackage{latexsym}
\usepackage{wrapfig}
\usepackage{tikz}
\usetikzlibrary{knots}
\usetikzlibrary{tikzmark,decorations.pathreplacing}
\usepackage[margin=1in]{geometry}

\usepackage{thmtools}
\declaretheoremstyle[numbered=no]{probs}
\declaretheorem[
	name=Theorem,numberwithin=section
	]{thm}
\declaretheorem[
	name=Lemma,
	sibling=thm,
	]{lem}

\declaretheorem[
	name=Proposition,
	sibling=thm,
	]{prop}
\declaretheorem[
	name=Corollary,
	sibling=thm,
	]{cor}

\declaretheorem[
	name=Definition,
	sibling=thm
	]{defin}

\declaretheorem[
	name=Example,
	sibling=thm
	]{exam}
 \declaretheorem[
	name=Conjecture,
	sibling=thm
	]{conj}

\declaretheorem[name=Theorem 2.5, style=probs, sibling=thm]{Theorem2.5}
\newcommand{\defeq}{\overset{\operatorname{def}}{=}}
\newcommand{\Symm}{\mathbf{Sym}_{X_1,...,X_{n-1}}}
\newcommand{\Asymm}{\mathbf{ASym}_{X_1,...,X_{n-1}}}
\newcommand{\SymmN}{\mathbf{Sym}_{X_1,...,X_{n}}}
\newcommand{\AsymmN}{\mathbf{ASym}_{X_1,...,X_{n}}}
\newcommand{\Sign}{\mathrm{sgn}}
\newcommand{\wt}{\mathrm{Wt}}
\newcommand{\pfaf}[1]{\underset{1\leq k_2<k_1\leq #1}{\mathrm{Pf}}}
\usepackage{amsmath,xcolor}
\let\oldoverrightarrow\overrightarrow
\renewcommand{\overrightarrow}[2][]{{%
  \if$#1$\else\color{#1}\fi% Optional argument given...
   \oldoverrightarrow{\color{black}#2}%
}}
\let\oldunderleftarrow\underleftarrow
\renewcommand{\underleftarrow}[2][]{{%
  \if$#1$\else\color{#1}\fi% Optional argument given...
   \oldunderleftarrow{\color{black}#2}%
}}
\title{Alternating sign pentagons and Magog pentagons}
\author[1]{Moritz Gangl\fnref{fn1}}
\ead{moritz.gangl@univie.ac.at}
\fntext[fn1]{Email: moritz.gangl@univie.ac.at}
\affiliation[1]{organization={Fakultät für Mathematik, Universität Wien},
addressline={Oskar-Morgenstern-Platz 1},
postcode={1090 Wien},
city={Vienna},
country={Austria}}

\date{March 2023}
\tnotetext[t1]{The author acknowledges the financial support from the Austrian Science Foundation FWF, grant P34931. The author would like to thank Ilse Fischer for her excellent supervision and support, as well as the anonymous referees for their helpful comments.}
\begin{document}

\begin{abstract}
    Alternating sign triangles were introduced by Ayyer, Behrend and Fischer in 2016 and it was proven that there is the same number of alternating sign triangles with $n$ rows as there is of $n\times n$ alternating sign matrices. Later on Fischer gave a refined enumeration of alternating sign triangles with respect to a statistic $\rho$, which has the same distribution as the position of the unique $1$ in the top row of an alternating sign matrix, by connecting alternating sign triangles to $(n,n)$-Magog trapezoids for which such a statistic is known. We introduce two more statistics counting the all $0$-columns on the left and right in an alternating sign triangle yielding objects we call alternating sign pentagons. We then show the equinumeracy of these alternating sign pentagons with Magog pentagons of a certain shape taking into account the statistic $\rho$. Furthermore we deduce a generating function of these alternating sign pentagons with respect to the statistic $\rho$ in terms of a Pfaffian and consider the implications of our new results for some open conjectures. In particular we conjecture a refined equinumerousity between our Magog pentagons and Gog pentagons of a certain shape.
\end{abstract}


\maketitle
\section{Introduction}
An \emph{alternating sign matrix} (ASM) of order $n$, as introduced by Robbins and Rumsey (see \cite{DeterminantsAndAlternatingSignMatrices}), is an $n\times n$ matrix with entries in $\{-1,0,1\}$ such that all column- and row-sums are equal to $1$ and the non-zero entries alternate in sign along rows and columns. For example the matrix below is an alternating sign matrix of order $7$.
$$\begin{pmatrix}
    0&0&0&1&0&0&0\\
    0&1&0&-1&1&0&0\\
    1&-1&0&1&-1&1&0\\
    0&0&1&-1&0&0&1\\
    0&1&-1&1&0&0&0\\
    0&0&1&-1&1&0&0\\
    0&0&0&1&0&0&0
\end{pmatrix}$$
They conjectured that the set of ASMs of order $n$ is enumerated by the following beautiful product formula.
$$\prod_{i=0}^{n-1}\frac{(3i+1)!}{(n+i)!}$$
In 1994 Zeilberger \cite{ProofOfTheAlternatingSignMatrixConjecture} gave the first proof of this conjecture by showing that for all $n\geq 1$ the number of $(n,n)$\emph{-Gog trapezoids}, which are in bijective correspondence with alternating sign matrices (ASMs) of order $n$, is equal to the number of $(n,n)$\emph{-Magog trapezoids}. It relied on a result of Andrews \cite{PlanePartitionsVTheTSSCPPconjecture} enumerating the set of \emph{totally symmetric self complementary plane partitions}
(TSSCPPs) in a $2n\times 2n\times 2n$ box, which are in bijective correspondence with $(n,n)$-Magog trapezoids. However, Zeilberger was able to prove even more, he actually showed that for all $n \geq k\geq 1$, the number of $(n,k)$-Gog trapezoids is equal to the number of $(n,k)$-Magog trapezoids. A more general conjecture for $(m,n,k)$-Gog trapezoids and $(m,n,k)$-Magog trapezoids, which is still open, was posed by Krattenthaler in \cite[Conjecture 5]{PlanePartitionsInTheWorkOfRichardStanleyAndHisSchool}.

Over 20 years later, Ayyer, Behrend and Fischer \cite{ExtremeDiagonallyAndAntidiagonallySymmetricAlternatingSignMatricesOfOddOrder} introduced \emph{alternating sign triangles} (ASTs), another class of objects enumerated by the same sequence. In 2019 Fischer \cite{FischerEnumerationOfAlternatingSignTrianglesUsingAConstantTermApproach} came up with an alternative proof of this enumeration formula, using a constant term approach which she developed over the years \cite{TheNumberOfMonotoneTrianglesWithPrescribedBottomRow,ANewProofOfTheRefinedAlternatingSignMatrixTheorem,TheOperatorFormulaForMonotoneTriangles,RefinedEnumerationsOfAlternatingSignMatricesMonotoneDMTrapezoidsWithPrescribedTopAndBottomRow,ShortProofOfTheASMTheoremAvoidingTheSixVertexModel}. This way, she was able to relate ASTs to $(n,n)$-Magog trapezoids. In the same paper she stated a conjecture by Behrend, relating ASMs with a certain number of all zero NW-SE diagonals, counted from the top right corner, and all zero NE-SW diagonals, counted from the top left corner, to \emph{alternating sign pentagons} (ASPs), i.e. ASTs with a certain number of all zero columns on the right and on the left, including two further statistics on both sides. The above mentioned ASMs with cut off top corners are in easy bijection with a class of objects called \emph{Gog pentagons}, a generalisation of Gog trapezoids, and Gog trapezoids in turn are equinumerous with Magog trapezoids as shown by Zeilberger \cite{ProofOfTheAlternatingSignMatrixConjecture}. While already in 2016 Biane and Cheballah \cite{GogAndGogampentagons} introduced \emph{GOGAm pentagons}, a family of objects, which they conjectured to be equinumerous with the family of Gog pentagons, this is still an open problem. For more details on this see Section \ref{Section Consequences and conjectures}. In this paper we will define \emph{Magog pentagons}, not related to the GOGAm pentagons, and give a proof similar in style to Fischer's proof from 2019, relating the family of alternating sign pentagons mentioned in the conjecture by Behrend with the newly defined Magog pentagons. This shows that the newly introduced objects are a good candidate for the Gog pentagon equivalent on the Magog-side and yields an equivalent conjecture to the one formulated by Behrend. We will even be able to include a statistic on both sides for this result and as a corollary we will also prove part of a conjecture by Behrend. 
\section{Prerequisites and main results}\label{SectionPrerequisitesAndMainResults}
In 2017 Ayyer, Behrend and Fischer introduced alternating sign triangles in \cite{ExtremeDiagonallyAndAntidiagonallySymmetricAlternatingSignMatricesOfOddOrder} as defined below.
\begin{defin}
An \emph{alternating sign triangle} (AST) of order $n$ is a triangular array of the form 
    $$\begin{matrix}
    a_{1,1}&a_{1,2}&a_{1,3}&...&...&...&a_{1,2n-1}\\
    &a_{2,2}&a_{2,3}&...&...&a_{2,2n-2}&\\
    & &...&...&...& &\\
    & & &a_{n,n}& & &
    \end{matrix}$$
with $a_{i,j}\in\{-1,0,1\}$ such that
\begin{enumerate}
    \item the non-zero entries alternate along each row and column,
    \item the row sums are equal to $1$,
    \item and the top most non-zero entry of each column is equal to $1$, if it exists.
\end{enumerate}
A $1$-\emph{column} is a column of an AST which sums up to $1$. Sticking to the notation in \cite{FischerEnumerationOfAlternatingSignTrianglesUsingAConstantTermApproach}, we further differentiate between $10$-\emph{columns} and $11$-\emph{columns}. The former are those $1$-columns whose bottom entry is $0$, whereas in the latter the bottom entry is $1$. Furthermore we define for an AST $T$ the following statistic $\rho$ by
\begin{multline}
    \rho(T)=\#(\textrm{of }11\textrm{-columns to the left of the central column})\\
    +\#(\textrm{of }10\textrm{-columns to the right of the central column})+1.
\end{multline}
\end{defin}
We will use the following notation introduced by Fischer in \cite{FischerEnumerationOfAlternatingSignTrianglesUsingAConstantTermApproach}: label the columns of the AST from left to right, starting with $0$, skipping the central column and ending with $2n-3$. See for example the labels of the columns of the AST of order $4$ in blue below.
$$\begin{matrix}
0&0&0&1&0&0&0\\
 &1&0&-1&0&1& \\
 & &1&0&0& & \\ 
 & & &1& & &\\
\textcolor{blue}{0}&\textcolor{blue}{1}&\textcolor{blue}{2}& &\textcolor{blue}{3}&\textcolor{blue}{4}&\textcolor{blue}{5}
\end{matrix}$$
Note that the central column of an AST is always a $1$-column since its bottom entry $a_{n,n}$ is equal to $1$ and there are $n$ $1$-columns in any AST of order $n$. This is because by definition each of the $n$ rows of the AST sum to $1$ and hence by a double counting argument also exactly $n$ of the columns have to sum to $1$.
\begin{exam} The seven ASTs of order $3$ and their $\rho$-statistic are given by 
\begin{gather*}
    \begin{matrix}1&0&0&0&0\\&1&0&0&&\rho=3\\&&1&&\end{matrix}\quad\textrm{,}\quad
    \begin{matrix}1&0&0&0&0\\&0&0&1&&\rho=2\\&&1&&\end{matrix}\quad\textrm{,}\quad \begin{matrix}0&0&0&0&1\\&0&0&1&&\rho=1\\&&1&&\end{matrix}\quad\textrm{,}\quad \\
    \begin{matrix}0&0&0&0&1\\&1&0&0&&\rho=2\\&&1&&\end{matrix}\quad\textrm{,}\quad
    \begin{matrix}0&0&0&1&0\\&1&0&0&&\rho=3\\&&1&&\end{matrix}\quad\textrm{,}\quad
    \begin{matrix}0&1&0&0&0\\&0&0&1&&\rho=1\\&&1&&\end{matrix}\quad\textrm{,}\quad \\
    \begin{matrix}0&0&1&0&0\\&1&-1&1&&\rho=2\\&&1&&\end{matrix}\quad\textrm{,}\quad
\end{gather*}
\end{exam}
In the paper \cite{ExtremeDiagonallyAndAntidiagonallySymmetricAlternatingSignMatricesOfOddOrder} Ayyer, Behrend and Fischer showed that for any positive integer $n$ there is the same number of ASTs of order $n$ as there is of ASMs of order $n$. Similar to many other results relating equinumerous objects like TSSCPPs in a $2n\times 2n\times 2n$ box \cite{SelfComplementaryTotallySymmetricPlanePartitions} or descending plane partitions (DPPs) whose parts do not exceed $n$ \cite{AlternatingSignMatricesAndDescendingPlanePartitions} with ASMs, the proof is not bijective. Note that finding bijective proofs in this area is very hard, some say even impossible. For the reader interested in an overview on what has been tried we refer to the literature(see \cite{ASimpleExplicitBijectionBetweenN2GogAndMagogTrapezoids,GogAndMagogTrianglesAndTheSchutzenberger2LeftTrapezoids, GogAndMagogTrianglesAndTheSchutzenbergerInvolution, OnTheDoublyRefinedEnumerationOfAlternatingSignMatricesAndTotallySymmetricSelfComplementaryPlanePartitions, PlanePartitionsInTheWorkOfRichardStanleyAndHisSchool} for ASMs and TSSCCPs and \cite{ABijectiveProofOfTheASMTheoremPartI:TheOperatorFormula,TheFirstBijectiveProofOfTheAlternatingSignMatrixTheorem,TheFirstBijectiveProofOfTheRefinedASMTheorem,AlternatingSignMatricesWithReflectiveSymmetryAndPlanePartitions:n+3PairsOfEquivalentStatistics, ABijectiveProofOfTheASMTheoremPartII:ASMEnumerationAndASMDPPRelation,TheRelationBetweenAlternatingSignMatricesAndDescendingPlanePartitionsN+3PairsOfEquivalentStatistics} for ASMs and DPPs).

The main objects of interest in this paper are alternating sign pentagons, a generalisation of ASTs.
\begin{defin}\label{Def of ASP}
Let $n \in\mathbb{N}$ and $0\leq l\leq n-2< r\leq 2n-3$ such that $r-l> n-3$. An $(n,l,r)$-\emph{alternating sign pentagon} (ASP) is an AST of order $n$, where the positions $j_1,...,j_{n-1}$ of the $n-1$ non-central $1$-columns satisfy $l\leq j_1<...<j_{n-1}\leq r$.
\end{defin}
Note that if $r-l\leq n-3$ there do not exist any $(n,l,r)$-ASPs. 
This follows from the fact that if $r-(l-1)\leq n-2$, there cannot exist any $(n,l,r)$-ASPs, since then there are at most $r-(l-1)+1\leq n-1$ columns and so in particular at most $r-(l-1)+1\leq n-1$ $1$-columns, which is impossible as any such object is an AST of order $n$ and hence has to have exactly $n$ $1$-columns. 

Also note that then necessarily all the columns indexed by $0,...,l-1,r+1,...,2n-3$ only consist of $0$'s, so we can simply cut them off. To see this assume the contrary and consider the rightmost $0$-column not consisting of all $0$'s. Since it is not a $1$-column, the bottommost non-zero entry has to be $-1$, but then the row of this entry cannot sum to $1$ as there are no non-zero entries to the right of it. A similar argument works for the left side. This justifies also the name alternating sign pentagon. 
\begin{exam}
An $(6,3,8)$-ASP $T$ with $\rho(T)=3$.
\setcounter{MaxMatrixCols}{11}
$$\begin{matrix}
     0&0&0&0&0&1&0&0&0&0&0\\
     &0&0&0&1&-1&0&0&0&1&\\
     &&0&0&0&1&0&0&0&&\\
     &&&1&-1&0&0&1& &&\\
     &&& &1&-1&1& &&&\\
     && & & &1& & & & 
\end{matrix}
\begin{matrix}
     &&0&0&1&0&0&0&0\\
     &&0&1&-1&0&0&0&1\\
     &\leftrightarrow&0&0&1&0&0&0&\\
     &&1&-1&0&0&1& &&\\
     && &1&-1&1& &&&\\
     & & & &1& & & &  
\end{matrix}$$
\end{exam}

These objects were already hinted at in \cite{FischerEnumerationOfAlternatingSignTrianglesUsingAConstantTermApproach}.
Our main result relates the $(n,l,r)$-ASPs to combinatorial objects related to TSSCCPs, making them fit into the bigger picture connecting ASMs, ASTs, TSSCCPs, DPPs. It is well known (see for example \cite{PlanePartitionsInTheWorkOfRichardStanleyAndHisSchool,ProofOfTheAlternatingSignMatrixConjecture}) that TSSCCPs correspond bijectively to $(n,n)$-Magog trapezoids as defined below, where we borrow the terminology from \cite{ProofOfTheAlternatingSignMatrixConjecture}.
\begin{defin}
    An $(m,n,k)$-\emph{Magog trapezoid} is an array of positive integers consisting of the top $k$ rows of an array
    $$\begin{matrix}
        a_{1,1}&a_{2,1}&...&...&a_{n,1}\\
        a_{2,2}&a_{3,2}&...&a_{n,2}&\\
        ...&...&...& &\\
        a_{n-1,n-1}&a_{n,n-1}& & &\\
        a_{n,n}& & & &
    \end{matrix}$$
such that entries along rows are weakly increasing, entries along columns are weakly decreasing, and such that the entries in the first row are bounded by $a_{k,1}\leq m+k$. We will suppress the $m$ whenever it is zero and simply write $(n,k)$-Magog trapezoid for any $(0,n,k)$-Magog trapezoid.
\end{defin}
We now define a new object generalising the $(m,n,k)$-Magog trapezoids.
\begin{defin}
    An $(m,n,k,d)$-\emph{Magog pentagon} consists of the first $d$ NW-SE diagonals, where we count the diagonals from top right to bottom left, of an $(m,n,k)$-Magog trapezoid. In other words, it is an array of positive integers consisting of the intersection of the top $k$ rows and the first $d$ NW-SE diagonals, where we count the diagonals from top right to bottom left, of an array
    $$\begin{matrix}
         a_{1,1}&a_{2,1}&...&...&a_{n,1}\\
        a_{2,2}&a_{3,2}&...&a_{n,2}&\\
        ...&...&...& &\\
        a_{n-1,n-1}&a_{n,n-1}& & &\\
        a_{n,n}& & & &
    \end{matrix}$$
such that entries along rows are weakly increasing, entries along columns are weakly decreasing, and such that the entries in the first row are bounded by $a_{k,1}\leq m+k$. We will suppress the $m$ whenever it is zero and simply write $(n,k,d)$-Magog pentagon for any $(0,n,k,d)$-Magog pentagon. We define a statistic $\tau$ on $(m,n,k,d)$-Magog pentagons $P$ by \begin{equation}
    \tau(P)=n+\sum_{i=1}^k(a_{n-1,i}-a_{n,i}).
\end{equation}
\end{defin}
Note that since $(m,n,k,2n-1)$-Magog pentagons are simply $(m,n,k)$-Magog trapezoids, this statistic is also defined on Magog trapezoids. This statistic was introduced by Mills, Robbins and Rumsey on Magog triangles \cite{SelfComplementaryTotallySymmetricPlanePartitions}.

Also note that these objects are only of actual pentagonal shape if $ \max\{n+1,2k+1\}\leq d\leq n+k-2$ and they degenerate to $(m,n,k)$-Magog trapezoids if $n+k-1\leq d$. Nevertheless we go with the name \emph{Magog pentagons} because of the connections to Gog pentagons explained in Section \ref{Section Consequences and conjectures}. 

\begin{exam}A $(10,4,11)$-Magog pentagon $P$ with weight $\tau(P)=5$.
$$\begin{matrix}
1&2&2&4&5&6&7&7&8&9\\
1&2&2&4&5&5&5&5&7\\
&2&2&4&4&4&4&5\\
& &2&2&2&2&3\\
\end{matrix}$$
\end{exam}

Our main result explains the connection between weighted $(n,l,r)$-ASPs and certain weighted Magog pentagons, generalising the relation of ASTs and $(n,n)$-Magog trapezoids proven by Fischer in \cite{FischerEnumerationOfAlternatingSignTrianglesUsingAConstantTermApproach}.
\begin{thm}\label{RelationASTsMagogPentagons}
Let $n\in\mathbb{N}$, $1\leq p\leq n$ and $0\leq l\leq n-2< r\leq 2n-3$, such that $l+r< 2n-2$. The set of $(n,l,r)$-ASPs $P$ with $\rho(P)=p$ is equinumerous with the set of $(n,r+2-n,r-l)$-Magog pentagons $M$ with $\tau(M)=p$.
\end{thm}
Although we have some restriction on $l$ and $r$ in terms of $l+r<2n-2$ this statement still exhausts the full class of ASPs due to a simple reason of symmetry. To see this note that for any pair $0\leq l\leq n-2< r\leq 2n-3$ we either have $l+r< 2n-2$ or $(2n-3-r)+(2n-3-l)< 2n-2$, so always either the $(n,l,r)$-ASPs or the $(n,2n-3-r,2n-3-l)$-ASPs satisfy the condition of Theorem \ref{RelationASTsMagogPentagons}. 

Next observe that ASTs are symmetric with respect to reflection along the central column. This reflection turns $(n,l,r)$-ASPs into $(n,2n-3-r,2n-3-l)$-ASPs and all of these subclasses exhaust the full class of ASPs. It remains to check the effect of the reflection on the $\rho$-statistic. Since columns on the left of the central column turn into columns on the right of the central column by reflection and vice versa, we have that $$n+1=\rho(T)+\rho(T'),$$
for any AST $T$, where $T'$ denotes the reflection of $T$ along its central column, since the central column is counted twice and all other $1$-columns are counted exactly once. So we conclude that the number of $(n,l,r)$-ASPs $P$ with $\rho(P)=p$ is the same as the number of $(n,2n-3-r,2n-3-l)$-ASPs $T$ with $\rho(T)=n+1-p$. 

Note that the case $l+r=2n-3$ is special, as then $l=2n-3-r$ and also $(2n-3-r)+(2n-3-l)=2n-3$. Such $(n,l,r)$-ASPs are clearly symmetric with respect to reflection along the central column, hence the classes of $(n,l,r)$-ASPs $P$ of order $n$ with $\rho(P)=p$ and those $(n,l,r)$-ASPs $T$ of order $n$ with $\rho(T)=n+1-p$ are equinumerous. See, for example, the three $(3,1,2)$-ASPs and their $\rho$-statistics below.
\begin{gather*}
    \begin{matrix}0&0&1\\1&0&0&\\&1&\\&\rho=3&\end{matrix}\quad\textrm{,}\quad
    \begin{matrix}0&1&0\\1&-1&1&\\&1&\\&\rho=2&\end{matrix}\quad\textrm{,}\quad 
    \begin{matrix}1&0&0\\0&0&1&\\&1&\\&\rho=1&\end{matrix}
\end{gather*}

\begin{exam}Consider for $n=3$, $l=0$, $r=2$ the five $(3,0,2)$-ASPs, $(3,1,3)$-ASPs and their $\rho$-statistic and the five  $(3,1,2)$-Magog pentagons with their $\tau$-statistic below and observe the symmetry explained above and the statement of Theorem \ref{RelationASTsMagogPentagons}. Note that these $(3,1,2)$-Magog pentagons are not of actual pentagonal shape, since $d<2k+1$ in this case.
\begin{gather*}
    \begin{matrix}1&0&0&0&\\&1&0&0&\\&&1&&\\&&\rho=3&&\end{matrix}\quad\textrm{,}\quad
    \begin{matrix}0&0&0&1&\\&1&0&0&\\&&1&&\\&&\rho=3&&\end{matrix}\quad\textrm{,}\quad
    \begin{matrix}0&0&1&0&\\&1&-1&1&\\&&1&&\\&&\rho=2&&\end{matrix}\quad\textrm{,}\quad \\
    \begin{matrix}1&0&0&0&\\&0&0&1&\\&&1&&\\&&\rho=2&&\end{matrix}\quad\textrm{,}\quad
    \begin{matrix}0&1&0&0&\\&0&0&1&\\&&1&&\\&&\rho=1&&\end{matrix}\quad\textrm{,}\quad\\
    \end{gather*}
    \begin{gather*}
    \begin{matrix}&0&0&0&1\\&0&0&1&\\&&1&&\\&&\rho=1&&\end{matrix}\quad\textrm{,}\quad
    \begin{matrix}&1&0&0&0\\&0&0&1&\\&&1&&\\&&\rho=1&&\end{matrix}\quad\textrm{,}\quad
    \begin{matrix}&0&0&0&1\\&1&0&0&\\&&1&&\\&&\rho=2&&\end{matrix}\quad\textrm{,}\quad\\
    \begin{matrix}&0&1&0&0\\&1&-1&1&\\&&1&&\\&&\rho=2&&\end{matrix}\quad\textrm{,}\quad
    \begin{matrix}&0&0&1&0\\&1&0&0&\\&&1&&\\&&\rho=3&&\end{matrix}\quad\textrm{,}\quad\\
    \begin{matrix}& & 1 &1 & \\& & \tau=3& &\end{matrix}\quad\textrm{,}\quad
    \begin{matrix}& & 2 &2 & \\& & \tau=3& &\end{matrix}\quad\textrm{,}\quad
    \begin{matrix}& & 2 &3 & \\& & \tau=2& &\end{matrix}\\
    \begin{matrix}& & 1 &2 & \\& & \tau=2& &\end{matrix}\quad\textrm{,}\quad
    \begin{matrix}& & 1 &3 & \\& & \tau=1& &\end{matrix}\quad\textrm{.}\quad
\end{gather*}
\end{exam}

We also obtain a formula for the generating function of these newly defined objects in terms of a Pfaffian. Note that the first condition here is the same as the one in Theorem \ref{RelationASTsMagogPentagons} and the second inequality is the one from Definition \ref{Def of ASP}.
\begin{thm}\label{PfaffianCor}
Let $n\in\mathbb{N}$, $1\leq p\leq n$ and $0\leq l\leq n-2< r\leq 2n-3$, such that $l+r<2n-2$ and $r-l>n-3$.  The generating function for the set of $(n,r+2-n,r-l)$-Magog pentagons $P$ with respect to $\tau$ is given by
\begin{equation}
    t\times \pfaf{n-1}\Bigg(\sum_{1\leq e_1<e_2\leq r+1}\det\Big(\scalebox{0.9}{$t\binom{k_j-1}{e_i-k_j}-t\binom{k_j-1}{r-e_i-l+2n-1-k_j}+\binom{k_j-1}{e_i-1-k_j}-\binom{k_j-1}{r-e_i-l+2n-k_j}$}\Big)_{i,j=1}^2\Bigg),
\end{equation}
if $n$ is odd. If $n$ is even, we have to add an $n$-th column to the Pfaffian with entries $a_{j,n}=\sum_{1\leq e_1\leq r+1}t\binom{j-1}{e_1-j}-t\binom{j-1}{r-e_1-l+2n-1-j}+\binom{j-1}{e_1-1-j}-\binom{j-1}{r-e_1-l+2n-j}$ for $j\in[n]$.
\end{thm}

\medskip
{\bf Outline of the paper.} The proof of Theorem \ref{RelationASTsMagogPentagons} occupies Sections \ref{SectionIntermediateResult} and \ref{Section Proof of Theorem}. In Section \ref{SectionIntermediateResult} we derive a constant term formula for the number of $(n,l,r)$-ASPs $P$ with $\rho(P)=p$ and evaluate this constant term. In Section \ref{Section Proof of Theorem} we interpret this constant term expression as tuples of non-intersecting lattice paths via the Lindström-Gessel-Viennot Lemma and provide a weight preserving bijection between these tuples of paths and the $(n,r+2-n,r-l)$-Magog pentagons mentioned in Theorem \ref{RelationASTsMagogPentagons} to finish the proof. In Section \ref{Section Proof of Pfaffian} we make use of a result by Stembridge and the constant term expression derived in Section \ref{SectionIntermediateResult} again, to give a proof of Theorem \ref{PfaffianCor}. Section \ref{SectionIntermediateResult} is dedicated to a detailed discussion of the connection of Magog pentagons to Gog pentagons and GOGAm pentagons. We use Zeilbergers famous Lemma 1 \cite[Lemma 1]{ProofOfTheAlternatingSignMatrixConjecture} to deduce a proof of part of a conjecture by Behrend relating ASMs and ASTs and provide some conjectures on Magog pentagons of our own. 
\section{An intermediate result}\label{SectionIntermediateResult}
As mentioned in the introduction, the proof of Theorem \ref{RelationASTsMagogPentagons} makes use of a constant-term approach developed by Fischer in \cite{TheNumberOfMonotoneTrianglesWithPrescribedBottomRow,ANewProofOfTheRefinedAlternatingSignMatrixTheorem,TheOperatorFormulaForMonotoneTriangles,RefinedEnumerationsOfAlternatingSignMatricesMonotoneDMTrapezoidsWithPrescribedTopAndBottomRow,ShortProofOfTheASMTheoremAvoidingTheSixVertexModel} and used by her in 2019 to give another proof of the equinumeracy of ASTs of order $n$ and TSSCPPs in an $2n\times 2n\times 2n$ box in \cite{FischerEnumerationOfAlternatingSignTrianglesUsingAConstantTermApproach}. In the same paper she was able to prove a conjecture formulated in \cite{ExtremeDiagonallyAndAntidiagonallySymmetricAlternatingSignMatricesOfOddOrder}, relating the $\rho$ statistic on the AST side with the position of the unique $1$ in the top row of an ASM. We state \cite[Theorem 2]{FischerEnumerationOfAlternatingSignTrianglesUsingAConstantTermApproach}, one of the main results of Fischer's paper, formulated in a way that we will make use of.
\begin{thm}[\cite{FischerEnumerationOfAlternatingSignTrianglesUsingAConstantTermApproach}, Theorem 2]\label{FischerTheorem2}
Let $n\in\mathbb{N}$ and $1\leq p\leq n$. The number of ASTs $T$ of order $n$ with $\rho(T)=p$ and non central $1$-columns in positions $0\leq j_1<j_2<...j_{n-1}\leq 2n-3$ is given by the constant term with respect to $X_1,...,X_{n-1},t$ of
\begin{equation}\label{ConstantTermNumberOfJ1...JN-1ASTS}
    t^{-p+1}\prod_{i=1}^{n-1}(t+X_i)X_i^{-j_i}\prod_{1\leq i<j\leq n-1}(1+X_i+X_iX_j)(X_j-X_i).
\end{equation}
\end{thm}
 As an immediate consequence of the above Theorem \ref{FischerTheorem2}, we obtain a constant term expression for the number of $(n,l,r)$-ASPs $P$ with $\rho(P)=p$.
\begin{cor}\label{UnevaluatedCTOfLRRestrictedASTSCor}
Let $n\in\mathbb{N}$, $1\leq p\leq n$ and $0\leq l\leq n-2< r\leq 2n-3$. The number of $(n,l,r)$-ASPs $P$ with $\rho(P)=p$ is given by the constant term with respect to $X_1,...,X_n,t$ of
\begin{equation}
    t^{1-p}\sum_{l\leq j_1<j_2<...<j_{n-1}\leq r}\prod_{i=1}^{n-1}(t+X_i)X_i^{-j_i}\prod_{1\leq i<j\leq n-1}(1+X_i+X_iX_j)(X_j-X_i).
\end{equation}
\end{cor}
Our first step in the proof of Theorem \ref{RelationASTsMagogPentagons} is the following intermediate result giving a closed form for this constant term.
\begin{thm}\label{ThmConstantTermExpressionForLRRestrictedASTS}
Let $n\in\mathbb{N}$, $1\leq p\leq n$ and $0\leq l\leq n-2< r\leq 2n-3$. The number of $(n,l,r)$-ASPs $P$ with $\rho(P)=p$ is the constant term with respect to $t$ in
\begin{multline}\label{SumOfDeterminantsLRRestrictedASTFormula}
    t^{1-p}
    \sum_{1\leq e_1<e_2<...<e_{n-1}\leq r+1}\det\bigg(t\binom{j-1}{e_i-j}-t\binom{j-1}{r-e_i-l+2n-1-j}\\
    +\binom{j-1}{e_i-1-j}-\binom{j-1}{r-e_i-l+2n-j}\bigg)_{i,j=1}^{n-1}.
\end{multline}
\end{thm}
We define the \emph{symmetrizer} $\SymmN$ and \emph{antisymmetrizer} $\AsymmN$ of a function $f(X_1,...,X_{n})$ with respect to $X_1,...,X_{n}$.
\begin{gather*}
    \SymmN f(X_1,...,X_{n})\defeq \sum_{\sigma\in \mathcal{S}_{n}}f(X_{\sigma(1)},...,X_{\sigma(n)}),\\
    \AsymmN f(X_1,...,X_{n})\defeq \sum_{\sigma\in \mathcal{S}_{n}}\Sign(\sigma) f(X_{\sigma(1)},...,X_{\sigma(n)}).
\end{gather*}
We will make use of the following fact, telling us that for any function $s(X_1,...,X_{n})$ symmetric in $X_1,...,X_{n}$, we have that $$\SymmN s(X_1,...,X_{n})=n!s(X_1,...,X_{n}),$$
therefore since any constant function is symmetric, we have in particular that the constant term of a function $f(X_1,...,X_{n})$ with respect to $X_1,...,X_{n}$ is equal to the constant term of $\frac{1}{n!}\SymmN f(X_1,...,X_{n})$ with respect to $X_1,...,X_{n}$. We will refer to this as the ``symmetrizer trick".

Another important property which we will make use of frequently throughout the following calculations is that for a function $a(X_1,...,X_{n})$ antisymmetric in $X_1,...,X_{n}$ and a function $f(X_1,...,X_{n})$, we have that
\begin{multline}
    \SymmN\Big(f(X_1,...,X_{n})a(X_1,...,X_{n})\Big)\\
    =a(X_1,...,X_{n})\AsymmN\big(f(X_1,...,X_{n})\big).
\end{multline}
Furthermore we will also use that for any function $s(X_1,...,X_{n})$ symmetric in $X_1,...,X_{n}$ and a function $f(X_1,...,X_{n})$, we have that
\begin{multline}
    \SymmN\Big(f(X_1,...,X_{n})s(X_1,...,X_{n})\Big)\\
    =s(X_1,...,X_{n})\SymmN\big(f(X_1,...,X_{n})\big).
\end{multline}
One of the main ingredients of the proof is the following special case of Corollary 1.2 by Fischer from \cite{BoundedLittlewoodIdentityRelatedToAlternatingSignMatrices} obtained by setting $w=0$ in her result. 
\begin{lem}[\cite{BoundedLittlewoodIdentityRelatedToAlternatingSignMatrices}, Cor. 1.2]\label{BoundedLittlewoodIdentity}
For $n\in\mathbb{N}$ we have that
\begin{multline}
    \AsymmN\bigg[\prod_{1\leq i<j\leq n}(1+X_j+X_iX_j)\sum_{0\leq k_1<k_2<...<k_{n}\leq r}\prod_{i=1}^{n}X_{i}^{k_i}\bigg]\\
=\frac{\det\bigg(X_i^{j-1}(1+X_i)^{j-1}-X_i^{r+2n-j}(1+X_i^{-1})^{j-1}\bigg)_{i,j=1}^n}{\prod_{i=1}^{n}(1-X_i)\prod_{1\leq i<j\leq n}(1-X_jX_i)}.
\end{multline}
\end{lem}
Having collected the necessary ingredients, we can start with the proof of Theorem \ref{ThmConstantTermExpressionForLRRestrictedASTS}, so let $n\in\mathbb{N}$, $1\leq p\leq n$ and $0\leq l\leq n-2< r\leq 2n-3$. We meed to show that the number of $(n,l,r)$-ASPs $P$ with $\rho(P)=p$ is the constant term with respect to $t$ in
\begin{multline}\label{SumOfDeterminantsLRRestrictedASTFormula}
    t^{1-p}
    \sum_{1\leq e_1<e_2<...<e_{n-1}\leq r+1}\det\bigg(t\binom{j-1}{e_i-j}-t\binom{j-1}{r-e_i-l+2n-1-j}\\
    +\binom{j-1}{e_i-1-j}-\binom{j-1}{r-e_i-l+2n-j}\bigg)_{i,j=1}^{n-1}.
\end{multline} By Corollary \ref{UnevaluatedCTOfLRRestrictedASTSCor}, we know that the number of all $(n,l,r)$-ASPs $T$ with $\rho(T)=p$ is given by the constant term with respect to $X_1,...,X_{n-1},t$ of
\begin{multline}
    t^{1-p}\sum_{l\leq j_1<j_2<...<j_{n-1}\leq r}\prod_{i=1}^{n-1}(t+X_i)X_i^{-j_i}\prod_{1\leq i<j\leq n-1}(1+X_i+X_iX_j)(X_j-X_i)\\
    =t^{1-p}\prod_{i=1}^{n-1}X_i^{-l}\sum_{0\leq j_1<j_2<...<j_{n-1}\leq r-l}\prod_{i=1}^{n-1}(t+X_i)X_i^{-j_i}\prod_{1\leq i<j\leq n-1}(1+X_i+X_iX_j)(X_j-X_i).
\end{multline}
Since we are only looking for the constant term with respect to $X_1,...,X_{n-1}$ we may substitute $X_i=X_{n-i}$ to transform the expression as below.
\begin{multline*}
t^{1-p}\prod_{i=1}^{n-1}X_{n-i}^{-l}\sum_{0\leq j_1<j_2<...<j_{n-1}\leq r-l}\prod_{i=1}^{n-1}(t+X_{n-i})X_{n-i}^{-j_i}\\
\times \prod_{1\leq i<j\leq n-1}(1+X_{n-i}+X_{n-i}X_{n-j})(X_{n-j}-X_{n-i})\\
=t^{1-p}\prod_{i=1}^{n-1}X_i^{-l}\sum_{0\leq j_1<j_2<...<j_{n-1}\leq r-l}\prod_{i=1}^{n-1}(t+X_i)X_{n-i}^{-j_i}\prod_{1\leq i<j\leq n-1}(1+X_j+X_iX_j)(X_i-X_j).
\end{multline*}
Now we make use of the symmetrizer trick and consider the constant term of the expression below instead.
\begin{multline*}
    \frac{t^{1-p}}{(n-1)!}\Symm\bigg[\sum_{0\leq j_1<j_2<...<j_{n-1}\leq r-l}\prod_{i=1}^{n-1}(t+X_i)X_{n-i}^{-j_i}\prod_{i=1}^{n-1}X_i^{-l}\\
    \times \prod_{1\leq i<j\leq n-1}(1+X_j+X_iX_j)(X_i-X_j)\bigg]
\end{multline*}
Observing that $\prod_{i=1}^{n-1}(t+X_i)X_i^{-l}$ is symmetric and $\prod_{1\leq i<j\leq n-1}(X_i-X_j)$ is antisymmetric we can rewrite this as follows.
\begin{multline*}
    \frac{t^{1-p}}{(n-1)!}\prod_{i=1}^{n-1}(t+X_i)X_i^{-l}\\
    \times\Symm\bigg[\sum_{0\leq j_1<j_2<...<j_{n-1}\leq r-l}\prod_{i=1}^{n-1}X_{n-i}^{-j_i}\prod_{1\leq i<j\leq n-1}(1+X_j+X_iX_j)(X_i-X_j)\bigg]\\
    =\frac{t^{1-p}}{(n-1)!}\prod_{i=1}^{n-1}(t+X_i)X_i^{-l}\prod_{1\leq i<j\leq n-1}(X_i-X_j)\\
    \times \Asymm\bigg[\prod_{1\leq i<j\leq n-1}(1+X_j+X_iX_j)\sum_{0\leq j_1<j_2<...<j_{n-1}\leq r-l}\prod_{i=1}^{n-1}X_{n-i}^{-j_i}\bigg]
\end{multline*}
Noting that
\begin{multline*}
    \sum_{0\leq j_1<j_2<...<j_{n-1}\leq r-l}\prod_{i=1}^{n-1}X_{n-i}^{-j_i}=
    \prod_{i=1}^{n-1}X_{n-i}^{-(r-l)}\sum_{0\leq j_1<j_2<...<j_{n-1}\leq r-l}\prod_{i=1}^{n-1}X_{n-i}^{r-l-j_i}\\
    =\prod_{i=1}^{n-1}X_{i}^{l-r}\sum_{0\leq k_1<k_2<...<k_{n-1}\leq r-l}\prod_{i=1}^{n-1}X_{i}^{k_i}.
\end{multline*}
and since $\prod_{i=1}^{n-1}X_{i}^{l-r}$ is symmetric, we see that the expression is equal to
\begin{multline*}
    \frac{t^{1-p}}{(n-1)!}\prod_{i=1}^{n-1}(t+X_i)X_{i}^{-r}\prod_{1\leq i<j\leq n-1}(X_i-X_j)\\
    \times\Asymm\bigg[\prod_{1\leq i<j\leq n-1}(1+X_j+X_iX_j)\sum_{0\leq k_1<k_2<...<k_{n-1}\leq r-l}\prod_{i=1}^{n-1}X_{i}^{k_i}\bigg].
\end{multline*}
Using Lemma \ref{BoundedLittlewoodIdentity} for $n-1$, we see that this is equal to
\begin{multline*}
    \frac{t^{1-p}}{(n-1)!}\prod_{i=1}^{n-1}(t+X_i)X_{i}^{-r}\prod_{1\leq i<j\leq n-1}(X_i-X_j)\\
    \times \frac{\det\bigg(X_i^{j-1}(1+X_i)^{j-1}-X_i^{r-l+2n-2-j}(1+X_i^{-1})^{j-1}\bigg)_{i,j=1}^{n-1}}{\prod_{i=1}^{n-1}(1-X_i)\prod_{1\leq i<j\leq n-1}(1-X_jX_i)}.
\end{multline*}
Now we make use of the following identity
\begin{multline*}
\Asymm\bigg[\prod_{i=1}^{n-1}X_i^{i-1}\Big(1-\prod_{j=i}^{n-1}X_j\Big)^{-1}\bigg]=\prod_{i=1}^{n-1}(1-X_i)^{-1}\prod_{1\leq i<j\leq n-1}\frac{X_j-X_i}{(1-X_jX_i)}\\
=(-1)^{\binom{n-1}{2}}\prod_{i=1}^{n-1}(1-X_i)^{-1}\prod_{1\leq i<j\leq n-1}\frac{X_i-X_j}{(1-X_jX_i)},
\end{multline*}
which can be found in \cite{ProofOfTheAlternatingSignMatrixConjecture} as Subsublemma 1.1.3 and yields for our expression from above
\begin{multline*}
    \frac{t^{1-p}}{(n-1)!}\det\bigg(X_i^{j-1}(1+X_i)^{j-1}-X_i^{r-l+2n-2-j}(1+X_i^{-1})^{j-1}\bigg)_{i,j=1}^{n-1}\\
    \times\prod_{i=1}^{n-1}(t+X_i)X_{i}^{-r}(1-X_i)^{-1}\prod_{1\leq i<j\leq n-1}\frac{(X_i-X_j)}{(1-X_jX_i)}\\
    =\frac{t^{1-p}(-1)^{\binom{n-1}{2}}}{(n-1)!}\det\bigg(X_i^{j-1}(1+X_i)^{j-1}-X_i^{r-l+2n-2-j}(1+X_i^{-1})^{j-1}\bigg)_{i,j=1}^{n-1} \\\times \prod_{i=1}^{n-1}(t+X_i)X_{i}^{-r}\Asymm\bigg[\prod_{i=1}^{n-1}X_i^{i-1}\Big(1-\prod_{j=i}^{n-1}X_j\Big)^{-1}\bigg].
\end{multline*}
Since the determinant is antisymmetric in the $X_i$, we can push the determinant into the antisymmetrizer and this transforms it into a symmetrizer.
\begin{multline*}
\frac{t^{1-p}(-1)^{\binom{n-1}{2}}}{(n-1)!}\Symm\bigg[\prod_{i=1}^{n-1}(t+X_i)X_{i}^{-r}X_i^{i-1}\Big(1-\prod_{j=i}^{n-1}X_j\Big)^{-1}\\
\times\det\bigg(X_i^{j-1}(1+X_i)^{j-1}-X_i^{r-l+2n-2-j}(1+X_i^{-1})^{j-1}\bigg)_{i,j=1}^{n-1}\bigg]
\end{multline*}
Again making use of the symmetrizer trick, the constant term of this expression is the same as the constant term of the expression below.
\begin{multline*}
    t^{1-p}(-1)^{\binom{n-1}{2}}\det\bigg(X_i^{j-1}(1+X_i)^{j-1}-X_i^{r-l+2n-2-j}(1+X_i^{-1})^{j-1}\bigg)_{i,j=1}^{n-1}\\
    \times\prod_{i=1}^{n-1}(t+X_i)X_{i}^{-r}X_i^{i-1}\Big(1-\prod_{j=i}^{n-1}X_j\Big)^{-1}
\end{multline*}
Now observe that
\begin{multline*}
    \prod_{i=1}^{n-1}X_i^{i-1-r}\Big(1-\prod_{j=i}^{n-1}X_j\Big)^{-1}=\prod_{i=1}^{n-1}X_i^{i-1-r}\Big(\sum_{k_i\geq 0}\prod_{j=i}^{n-1}X_j^{k_i}\Big)=\sum_{k_1,...,k_{n-1}\geq 0}\prod_{i=1}^{n-1}X_i^{i-1-r}\prod_{j=i}^{n-1}X_j^{k_i}\\
    =\sum_{k_1,...,k_{n-1}\geq 0}\prod_{i=1}^{n-1}X_i^{k_1+...+k_i+i-1-r}=\sum_{0\leq b_1<b_2<...<b_{n-1}}\prod_{i=1}^{n-1}X_i^{b_i-r},
\end{multline*}
hence using the multilinearity of the determinant, the above is equal to
\begin{multline*}
    t^{1-p}(-1)^{\binom{n-1}{2}}\sum_{0\leq b_1<b_2<...<b_{n-1}}\det\bigg(X_i^{b_i-r}(t+X_i)\\\times\big(X_i^{j-1}(1+X_i)^{j-1}-X_i^{r-l+2n-2-j}(1+X_i^{-1})^{j-1}\big)\bigg)_{i,j=1}^{n-1}.
\end{multline*}
We only need to calculate the constant term with respect to $X_1,...,X_{n-1}$ of this sum of determinants, i.e. the constant term of every determinant. For this, note that the determinant is a sum of products of the entries, so again it suffices to calculate the constant term of each of the products. But since in each of the products exactly one entry from each row appears, the only way we can obtain a constant is by multiplying constants. This is because the variable $X_i$ appears only in row $i$, so if the factor from row $i$ in the product is not constant, then it contains a power of $X_i$, which cannot be cancelled with the other factors in the product. Hence it suffices to calculate the constant term of $$X_i^{b_i-r}(t+X_i)\big(X_i^{j-1}(1+X_i)^{j-1}-X_i^{r-l+2n-j}(1+X_i^{-1})^{j-1}\big)$$for arbitrary $i,j,b_i$. Using the binomial theorem, a calculation yields 
\begin{multline}
    t^{1-p}(-1)^{\binom{n-1}{2}}\sum_{0\leq b_1<b_2<...<b_{n-1}}\det\bigg(t\binom{j-1}{r+1-j-b_i}-t\binom{j-1}{b_i-l+2n-2-j}\\
    +\binom{j-1}{r-b_i-j}-\binom{j-1}{b_i-l-1+2n-j}\bigg)_{i,j=1}^{n-1}.
\end{multline}
Note that the sum is actually bounded, since for $b_i>r$, the $i$-th row is all $0$, hence we can substitute $s_i=r+1-b_{n-i}$ and reorder the columns via $\begin{pmatrix}1&2&...&n-2&n-1\\n-1&n-2&...&2&1\end{pmatrix}$, cancelling the sign, to obtain (\ref{SumOfDeterminantsLRRestrictedASTFormula}).

\section{Proof of Theorem \ref{RelationASTsMagogPentagons}}\label{Section Proof of Theorem}
To conclude the proof of our main result Theorem \ref{RelationASTsMagogPentagons} in this section, we need  to introduce Gelfand-Tsetlin-patterns, another class of objects closely related to Magog trapezoids and Magog pentagon.
\begin{defin}\label{DefAndBijGelfandTsetlinPattern}
Let $n\in\mathbb{N}$. A \emph{Gelfand-Tsetlin pattern} of order $n$ is an array of positive integers of the following form
    $$\begin{matrix}
     & & & &a_{1,1}& & & &\\
     & & &a_{2,2}& &a_{2,1}& & &\\
     & &...& &...& &...& &\\
     &a_{n-1,n-1}& &...& &...& &a_{n-1,1}&\\
     a_{n,n}& &a_{n,n-1}& &...& &a_{n,2}& &a_{n,1}
    \end{matrix}$$
such that $a_{i+1,j+1}\leq a_{i,j}\leq a_{i+1,j}$ for all entries. We label the $n$ NW-SE diagonals and the $2n-1$ columns of the Gelfand-Tsetlin-pattern of order $n$ from right to left.
\end{defin}
Throughout the paper we will always assume that $a_{k,1}\leq k$ for all $k\in [n]$ for all Gelfand-Tsetlin-patterns of order $n$. 

Now the first major part of the proof of Theorem \ref{RelationASTsMagogPentagons} is the following proposition, giving a combinatorial interpretation in terms of non-intersecting lattice paths for our constant term expression obtained in Theorem \ref{SumOfDeterminantsLRRestrictedASTFormula}. For this let us define the weighted set $\mathcal{P}_n^{l,r}$ of all non-intersecting $(n-1)$-tuples of lattice paths consisting of north- and east-steps in the integer lattice with
\begin{itemize}
    \item starting points $A_1=(1,-2), A_2=(2,-4),...,A_{n-1}=(n-1,-2(n-1))$,
    \item end-points in $B_1=(1,-1),B_2=(2,-2),...,B_{r+1}=(r+1,-r-1)$,
    \item the path starting at $A_{n-1}$ staying weakly above the line $y=x+l-r-2n+1$,
\end{itemize}
and where a path in a tuple is weighted by $t$ if and only if it ends with a north step. 
\begin{prop}\label{PropRelatingSumOfDetsAndLatticePaths}
Let $n\in\mathbb{N}$, $1\leq p\leq n$ and $0\leq l\leq n-2< r\leq 2n-3$, such that $l+r<2n-2$ and $r-l>n-3$.  Then
    \begin{multline*}
        \sum_{1\leq e_1<e_2<...<e_{n-1}\leq r+1}\det\bigg(t\binom{j-1}{e_i-j}-t\binom{j-1}{r-e_i-l+2n-1-j}\\
        +\binom{j-1}{e_i-1-j}-\binom{j-1}{r-e_i-l+2n-j}\bigg)_{i,j=1}^{n-1}
    \end{multline*}
is the generating function of the weighted set $\mathcal{P}_n^{l,r}$.
\end{prop}
Note that the first condition here is the same as the one in Theorem \ref{RelationASTsMagogPentagons} and the second inequality is the one from Definition \ref{Def of ASP}, so it makes sense to assume both of them.
\begin{proof}
The idea is to make use of the Lindström-Gessel-Viennot Lemma, see \cite{BinomialDeterminantsPathsAndHookLengthFormulae} and \cite{OnTheVectorRepresentationsOfInducedMatroids}, and the reflection principle for lattice paths \cite{ReflectionPrincipleOriginalSource}, see for example \cite{LatticePathEnumeration} for a nice overview. First of all observe the following:
\begin{enumerate}
    \item The points $B_m$ lie strictly above the line $L: y=x+l-r-2n$ for all $1\leq m\leq r+1$, since $l+r-2n+2<0$.
    \item The points $B_m^E\defeq(m-1,-m),B_m^N\defeq (m,-m-1)$ lie weakly above the line $L$ for all $1\leq m\leq r+1$, since $l+r-2n+2<0$.
    \item The point $R(B_m^N)\defeq(r-m-l+2n-1,l+m-r-2n)$ is the reflection of the point $(m,-m-1)$ along the line $L$ for all $1\leq m\leq r+1$,
    \item The point $R(B_m^E)\defeq(r-m-l+2n,l+m-r-2n-1)$ is the reflection of the point $(m-1,-m)$ along the line $L$ for all $1\leq m\leq r+1$,
    \item The intersection point $P\defeq(\frac{2n+r-l}{3},\frac{2(l-2n-r)}{3})$ of the line $y=-2x$ and $L$, lies below the line $y=-2(n-1)$, thus in particular below all $A_j$ for $j\in[n-1]$, since $\frac{2n+r-l}{3}> n-1$ which is equivalent to our assumption $r-l>n-3$.
\end{enumerate} See Figure 1 for the setup described above with parameters $n=10,r=12,l=3$.

\begin{figure}[ht]
            \begin{tikzpicture}[scale= 0.5]
     \draw[step=1 cm, color=gray!50, very thin] (0, 0) grid (20,-20);
    \draw[line width=1pt, rounded corners] (0,0)--(12.3,-12.3) (12.8,-12.8)--(15.3,-15.3) (15.8,-15.8)--(20,-20) (0,0)--(10,-20) (0,-18)--(7.8,-18) (8.8,-18)--(20,-18);
    \draw[line width=1pt, rounded corners, color=red] (9,-20)--(12.2,-16.8) (13,-16)--(20,-9) ;
    \draw[line width=1pt, rounded corners, color=brown] (15,-16)--(16,-16) (13,-14)--(13,-13) ;
    \draw[line width=1pt, rounded corners, color=blue] (16,-17)--(16,-16) (12,-13)--(13,-13);
    \fill (1,-2) circle (2pt);
    \fill (2,-4) circle (2pt);
    \fill (3,-6) circle (2pt);
    \fill (4,-8) circle (2pt);
    \fill (5,-10) circle (2pt);
    \fill (6,-12) circle (2pt);
    \fill (7,-14) circle (2pt);
    \fill (8,-16) circle (2pt);
    \fill (9,-18) circle (2pt);
    \fill[color=violet] (1,-1) circle (4pt);
    \fill[color=violet] (13,-13) circle (4pt);
    \fill[color=violet] (12,-13) circle (4pt);
    \fill[color=violet] (13,-14) circle (4pt);
    \fill[color=orange] (16,-16) circle (4pt);
    \fill[color=orange] (15,-16) circle (4pt);
    \fill[color=orange] (16,-17) circle (4pt);
    \fill[color=orange] (18,-18) circle (4pt);
    \fill[color=red] (29/3,-58/3) circle (4pt);
    \node[anchor=east] at (1,-2) {$A_1$};
    \node[anchor=east] at (2,-4) {$A_2$};
    \node[anchor=east] at (3,-6) {$A_3$};
    \node[anchor=east] at (4,-8) {$A_4$};
    \node[anchor=east] at (5,-10) {$A_5$};
    \node[anchor=east] at (6,-12) {$A_6$};
    \node[anchor=east] at (7,-14) {$A_7$};
    \node[anchor=east] at (8,-16) {$A_8$};
    \node[anchor=east] at (9,-18) {$A_9$};
    \node[anchor=west, color=orange] at (16,-16) {$R(B_{r+1})$};
    \node[anchor=west, color=blue] at (15.2,-16.5) {\small $1$};
    \node[anchor=west, color=brown] at (12.3,-13.5) {\small $t$};
    \node[anchor=west, color=blue] at (12,-12.6) {\small $1$};
    \node[anchor=west, color=brown] at (15,-15.6) {\small $t$};
    \node[anchor=east, color=orange] at (15.5,-16.5) {$R(B_{r+1}^N)$};
    \node[anchor=east, color=orange] at (16.5,-17.5) {$R(B_{r+1}^E)$};
    \node[anchor=east] at (0,-18) {$y=-2(n-1)$};
    \node[anchor=west] at (10,-21) {$y=-2x$};
    \node[anchor=east] at (21,-21) {$y=-x$};
    \node[anchor=west, color=violet] at (1,-0.8) {$B_1$};
    \node[anchor=west, color=violet] at (13,-13) {$B_{r+1}$};
    \node[anchor=east, color=violet] at (13.5,-14.5) {$B_{r+1}^N$};
    \node[anchor=east, color=violet] at (12,-13) {$B_{r+1}^E$};
    \node[anchor=west, color=orange] at (17.5,-17.5) {$R(B_{r+2-l})$};
    \node[anchor=east, color=red] at (9,-21) {$L: y=x+l-r-2n$};
    \node[anchor=east, color=red] at (29/3,-58/3) {$P$};

    \end{tikzpicture}
    \label{Figure1}
    \caption{Setup for the parameters $n=10,r=12,l=3$.}
    \centering
    \end{figure}
For now let us fix some $j\in[n-1]$ and $1\leq e_1<e_2<...<e_{n-1}\leq r+1$. We claim that upon using the reflection principle, we can interpret $$t\binom{j-1}{e_i-j}-t\binom{j-1}{r-e_i-l+2n-1-j}+\binom{j-1}{e_i-1-j}-\binom{j-1}{r-e_i-l+2n-j},$$ as the generating function of the set of lattice paths with starting point $A_j$ and end-point $B_{e_i}$ staying above the line $L$ shifted up by $1$, i.e. $y=x+l-r-2n+1$, where a path is weighted by $t$ if and only if it ends with a north-step. 

To see this, recall that by (3) and (4) the points $B_{e_i}^N$ and $R(B_{e_i}^N)$ as well as $B_{e_i}^E$ and $R(B_{e_i}^E)$ are the reflection of one another along the line $L$. By (2) the points $B_{e_i}^E, B_{e_i}^N$ lie weakly above the line $L$, hence by (3) and (4) the points $R(B_{e_i}^E), R(B_{e_i}^N)$ lie weakly below $L$. Recall that by (5) also $A_j$ lies above the line $L$, thus upon using the reflection principle twice, we see that the difference $t\binom{j-1}{e_i-j}-t\binom{j-1}{r-e_i-l+2n-1-j}$ weighted by $t$, gives the number of path starting at $A_j$ and ending at $B_{e_i}^N$ and staying weakly above the line $y=x+l-r-2n+1$, whereas the unweighted difference $\binom{j-1}{e_i-1-j}-\binom{j-1}{r-e_i-l+2n-j}$ yields the number of paths starting at $A_j$, ending at $B_{e_i}^E$ and staying weakly above the line $y=x+l-r-2n+1$. This shows our claim. See again \hyperref[Figure1]{Figure 1} above for $e_i=r+1$ and the distribution of the weights.

Now since by simple geometric reasons any path starting at $A_i$ and ending at $B_{e_k}$ intersects any path starting at $A_j$ and ending at $B_{e_l}$ for $i<j$ and $l<k$ or $j<i$ and $k<l$, we can conclude from the Lindström-Gessel-Viennot Lemma that 
$$\det\bigg(t\binom{j-1}{e_i-j}-t\binom{j-1}{r-e_i-l+2n-1-j}+\binom{j-1}{e_i-1-j}-\binom{j-1}{r-e_i-l+2n-j}\bigg)_{i,j=1}^{n-1}$$
is the generating function of all $(n-1)$-tuples of non-intersecting lattice paths from $A_j$ to $B_{e_j}$ staying weakly above the line $y=x+l-r-2n+1$, for $j\in[n-1]$, where a path is weighted by $t$ if and only if it ends with a north step. Note that since the paths are non-intersecting this condition of all paths staying weakly above the line $y=x+l-r-2n+1$ is equivalent to the lowest path, i.e. the path starting at $A_{n-1}$ and ending at $B_{e_{n-1}}$, staying weakly above the line $y=x+l-r-2n+1$. Now simply summing over all possible tuples $1\leq e_1<e_2<...<e_{n-1}\leq r+1$ yields the proposition.
\end{proof}
For completeness let us restate the unchecked part of the theorem before we finish the proof.

\emph{Let $n\in\mathbb{N}$, $1\leq p\leq n$ and $0\leq l\leq n-2< r\leq 2n-3$, such that $l+r<2n-2$ and $r-l>n-3$. The set of $(n,l,r)$-ASPs $P$ with $\rho(P)=p$ is equinumerous with the set of $(n,r+2-n,r-l)$-Magog pentagons $M$ with $\tau(M)=p$.}

After showing the following claim, we will be done by Proposition \ref{PropRelatingSumOfDetsAndLatticePaths}.\\

\begin{lem}Let $n\in\mathbb{N}$, $1\leq p\leq n$ and $0\leq l\leq n-2< r\leq 2n-3$, such that $l+r<2n-2$ and $r-l>n-3$. There is a weight preserving bijection between the set $\mathcal{P}_n^{l,r}$ and the set of $(n,r+2-n,r-l)$-Magog pentagons $M$ with $\tau(M)=p$.
\end{lem}

\begin{proof}
The bijection we use is actually well known and translates such tuples of paths into Gelfand-Tsetlin-patterns, see for example \cite{FischerEnumerationOfAlternatingSignTrianglesUsingAConstantTermApproach}. Nevertheless we will go through the details here because we have some additional properties for our paths and the $\rho$-statistic which we have to keep track of.

So let us consider the weighted set $\mathcal{P}_n^{l,r}$ as defined in Proposition \ref{PropRelatingSumOfDetsAndLatticePaths}.
Taking $T\in\mathcal{P}_n^{l,r}$, we start by shifting the path $P_j$ starting at $A_j$ by $(x,y)\mapsto (x-j,y+j)$ such that it now starts at $\Tilde{A}_j=(0,-j)$, for all $j\in[n-1]$. Now we obtain an $(n-1)$-tuple $\Tilde{T}$ of kissing lattice paths. By kissing we mean that the path $\Tilde{P}_j$ starting at $\Tilde{A}_j$ stays weakly below the path starting at $\Tilde{A}_{j-1}$. We claim that this is a bijection between the set $\mathcal{P}_n^{l,r}$ and the set $\Tilde{\mathcal{P}}_n^{l,r}$, where $\Tilde{\mathcal{P}}_n^{l,r}$ is the set of all $(n-1)$-tuples of kissing lattice paths $\Tilde{P}_j$ consisting of north- and east-steps in the integer square lattice with
\begin{itemize}
    \item starting points $\Tilde{A}_1=(0,-1),...,\Tilde{A}_{n-1}=(0,-n+1)$,
    \item end-points in $\Tilde{B}_0=(0,0),...,\Tilde{B}_{r-n+2}(r-n+2,n-r-2)$,
    \item the path $\Tilde{P}_{n-1}$ starting at $\Tilde{A}_{n-1}$ staying weakly above the line $y=x+l-r-1$,
\end{itemize}
where a path is weighted by $t$ if and only if it ends with a north-step. 

This is clear by definition of the map, except for the second and third property. Note that since the paths are kissing, the end-points of the paths $\Tilde{P}_1,...,\Tilde{P}_{n-1}$ have to stay weakly above the end-point of the path $\Tilde{P}_{n-1}$. Since the lowest possible end-point for $P_{n-1}$ is given by $B_{r+1}=(r+1,-r-1)$, we know by definition of the map that the lowest possible end-point for $\Tilde{P}_{n-1}$ is given by $\Tilde{B}_{r-n+2}=(r+1-(n-1),n-1-r-1)$. Similarly the highest possible end-point for $P_1$ is given by $B_1$ thus the highest possible end-point for $\Tilde{P}_1$ is given by $\Tilde{B}_0$, hence the second property follows. 

The argument for the third property is similar. Since $P_{n-1}$ has to stay weakly above the line $y=x+l-r-2n+1$ and we shift the path by $(x,y)\mapsto \Big(x-(n-1),y+(n-1)\Big)$, we know that $\Tilde{P}_{n-1}$ has to stay weakly above the line $y=x+l-r-1$. Consider for example the tuple $T\in\mathcal{P}_{10}^{1,12}$ of paths for $n=10,r=12,l=1$ on the left and its corresponding tuple $\Tilde{T}$ on the right in \hyperref[Figure2]{Figure 2}. To make the kissing paths easier to distinguish in the picture, we will colour them alternating in blue and orange and draw the blue lines dotted.
\begin{figure}[ht]
 \begin{tikzpicture}[scale= 0.2]
    \draw[step=1 cm, color=gray!50, very thin] (0, 0) grid (18,-18);
    \draw[line width=1pt, rounded corners] (0,0)--(18,-18) (0,0)--(9,-18);
    \draw[line width=1pt, rounded corners, color=green] (12,-18)--(18,-12) ;
    \draw[line width=1.5pt] (1,-2)--(1,-1) (2,-4)--(2,-3)--(3,-3) (3,-6)--(3,-5)--(4,-5)--(4,-4) (4,-8)--(5,-8)--(5,-6)--(6,-6) (5,-10)--(6,-10)--(6,-9)--(6,-7)--(7,-7) (6,-12)--(8,-12)--(8,-9)--(9,-9) (7,-14)--(10,-14)--(10,-10) (8,-16)--(8,-15)--(11,-15)--(11,-15)--(11,-12)--(12,-12) (9,-18)--(12,-18)--(12,-17)--(13,-17)--(13,-13);
    \draw[line width=1pt, rounded corners, color=red] (13,-18)--(18,-13) ;
    
    \fill (1,-2) circle (2pt);
    \fill (2,-4) circle (2pt);
    \fill (3,-6) circle (2pt);
    \fill (4,-8) circle (2pt);
    \fill (5,-10) circle (2pt);
    \fill (6,-12) circle (2pt);
    \fill (7,-14) circle (2pt);
    \fill (8,-16) circle (2pt);
    \fill (9,-18) circle (2pt);
    \fill[color=violet] (13,-13) circle (8pt);
    \node[anchor=east] at (1,-2) {$A_1$};
    \node[anchor=east] at (2,-4) {$A_2$};
    \node[anchor=east] at (3,-6) {$A_3$};
    \node[anchor=east] at (4,-8) {$A_4$};
    \node[anchor=east] at (5,-10) {$A_5$};
    \node[anchor=east] at (6,-12) {$A_6$};
    \node[anchor=east] at (7,-14) {$A_7$};
    \node[anchor=east] at (8,-16) {$A_8$};
    \node[anchor=east] at (9,-18) {$A_9$};
    \node[anchor=west, color=violet] at (12,-11) {\small $(r+1,-r-1)$};
    \node[anchor=west, color=red] at (18,-14) {\small $y=x+l-r-2n$};
    \node[anchor=west, color=green] at (18,-12) {\small $y=x+l-r-2n+1$};
\end{tikzpicture}
    \begin{tikzpicture}[scale=0.45]
    \draw[step=1 cm, color=gray!50, very thin] (0, 0) grid (9,-9);
    \draw[line width=3pt, rounded corners, color=orange](0,-1)--(0,0)  (0,-3)--(0,-2)--(1,-2)--(1,-1)  (0,-5)--(1,-5)--(1,-4)--(1,-2)--(2,-2)  (0,-7)--(3,-7)--(3,-3)
    (0,-9)--(3,-9)--(3,-8)--(4,-8)--(4,-4);
    \draw[line width=3pt, rounded corners, color=blue, dotted] (0,-2)--(0,-1)--(1,-1)  (0,-4)--(1,-4)--(1,-2)--(2,-2)  (0,-6)--(2,-6)--(2,-3)--(3,-3) 
    (0,-8)--(0,-7)--(3,-7)--(3,-4)--(4,-4);
    \draw[line width=1pt, rounded corners] (0,0)--(0,-9)  (0,0)--(9,-9) ;
    \draw[line width=1pt, rounded corners, color=red] (4,-9)--(9,-4);
    \draw[line width=1pt, rounded corners, color=green] (3,-9)--(9,-3);
    \node[anchor=west, color=red] at (9,-4) {$y=x+l-r-2$};
    \node[anchor=west, color=green] at (9,-3) {$y=x+l-r-1$};
    \node[anchor=east] at (0,-1) {$\Tilde{A}_1$};
    \node[anchor=east] at (0,-2) {$\Tilde{A}_2$};
    \node[anchor=east] at (0,-3) {$\Tilde{A}_3$};
    \node[anchor=east] at (0,-4) {$\Tilde{A}_4$};
    \node[anchor=east] at (0,-5) {$\Tilde{A}_5$};
    \node[anchor=east] at (0,-6) {$\Tilde{A}_6$};
    \node[anchor=east] at (0,-7) {$\Tilde{A}_7$};
    \node[anchor=east] at (0,-8) {$\Tilde{A}_8$};
    \node[anchor=east] at (0,-9) {$\Tilde{A}_9$};
    \end{tikzpicture}
    \label{Figure2}
    \caption{Example $T\in\mathcal{P}_{10}^{1,12}$ on the left and its corresponding tuple $\Tilde{T}$ on the right.}
    \centering
    \end{figure}

To translate this into Gelfand-Tsetlin-patterns we introduce a coordinate lattice on a Gelfand-Tsetlin-pattern of order $n$ as depicted in \hyperref[Figure3]{Figure 3}.
\begin{figure}[ht]
         \begin{tikzpicture}[scale=0.8]
         \draw[line width=1pt, color=gray!50] (6.5,0.5)--(12.5,-5.5) (5.5,-0.5)--(10.5,-5.5) (4.5,-1.5)--(8.5,-5.5) (3.5,-2.5)--(6.5,-5.5) (2.5,-3.5)--(4.5,-5.5) (8.5,0.5)--(2.5,-5.5) (9.5,-0.5)--(4.5,-5.5) (10.5,-1.5)--(6.5,-5.5) (11.5,-2.5)--(8.5,-5.5) (12.5,-3.5)--(10.5,-5.5);
          \node at (6.5,-0.5) {$a_{1,1}$};
          
          \node at (5.5,-1.5) {$a_{2,1}$};
          \node at (7.5,-1.5) {$a_{2,1}$};
          
          \node at (6.5,-2.5) {$...$};
          \node at (8.5,-2.5) {$...$};
          \node at (4.5,-2.5) {$...$};
          
          \node at (5.5,-3.5) {$...$};
          \node at (7.5,-3.5) {$a_{n-1,2}$};
          \node at (9.5,-3.5) {$a_{n-1,1}$};
          \node at (3.5,-3.5) {$a_{n-1,n-1}$};
          
          \node at (6.5,-4.5) {$...$};
          \node at (8.5,-4.5) {$a_{n,2}$};
          \node at (4.5,-4.5) {$...$};
          \node at (10.5,-4.5) {$a_{n,1}$};
          \node at (2.5,-4.5) {$a_{n,n}$};
          
          \node[color=blue] at (8.5,-0.5) {$\Hat{A}_{n-1}$};
          \node[color=blue] at (9.5,-1.5) {$\Hat{A}_{n-2}$};
          \node[color=blue] at (11.5,-3.5) {$\Hat{A}_{1}$};
          \fill[color=blue](7.5,-0.5) circle (3pt);
          \fill[color=blue](8.5,-1.5) circle (3pt);
          \fill[color=blue](10.5,-3.5) circle (3pt);
          
          \node[color=red] at (11.5,-5.5) {$\Hat{B}_{0}$};
          \node[color=red] at (9.5,-5.5) {$\Hat{B}_{1}$};
          \node[color=red] at (7.5,-5.5) {$\Hat{B}_{2}$};
          \node[color=red] at (5.5,-5.5) {$\Hat{B}_{r+2-n}$};
          \node[color=red] at (3.5,-5.5) {$\Hat{B}_{n-1}$};
          \fill[color=red](3.5,-4.5) circle (3pt);
          \fill[color=red](5.5,-4.5) circle (3pt);
          \fill[color=red](7.5,-4.5) circle (3pt);
          \fill[color=red](9.5,-4.5) circle (3pt);
          \fill[color=red](11.5,-4.5) circle (3pt);
    \end{tikzpicture}
    \centering
    \label{Figure3}
    \caption{The coordinate system on Gelfand-Tstelin-patterns.}
    \centering
    \end{figure}

Next we take a tuple $\Tilde{T}\in\Tilde{\mathcal{P}}_n^{l,r}$ and construct bijectively a Gelfand-Tsetlin-pattern $G(\Tilde{T})$ of order $n$. See below for an example. For this we take the path $\Tilde{P}_j$ starting at $\Tilde{A}_j$ and ending at say $\Tilde{B}_a$, and construct out of it a path $\Hat{P}_j$ starting at $\Hat{A}_j$ and ending at $\Hat{B}_a$ in the coordinate system above simply by making a south-east step in $\Hat{P}_j$ whenever $\Tilde{P}_j$ makes a north-step and a south-west step in $\Hat{P}_j$ whenever $\Tilde{P}_j$ makes an east-step. Then we fill in the number $j$, starting with $1$ all the way up to $n-1$, in all slots to the left of the path starting at $\Hat{A}_{n-j}$. The remaining empty slots will be filled with $n$. This process is clearly reversible and thus a bijection. See for example the tuple of paths $\Tilde{T}\in\Tilde{\mathcal{P}}_{10}^{1,12}$ from \hyperref[Figure2]{Figure 2} and its corresponding Gelfand-Tsetlin-pattern in \hyperref[Figure4]{Figure 4}.
\begin{figure}[ht]
   \begin{tikzpicture}[scale=0.5]
    \draw[step=1 cm, color=gray!50, very thin] (0, 0) grid (9,-9);
   \draw[line width=3pt, rounded corners, color=orange](0,-1)--(0,0)  (0,-3)--(0,-2)--(1,-2)--(1,-1)  (0,-5)--(1,-5)--(1,-4)--(1,-2)--(2,-2)  (0,-7)--(3,-7)--(3,-3)
    (0,-9)--(3,-9)--(3,-8)--(4,-8)--(4,-4);
    \draw[line width=3pt, rounded corners, color=blue, dotted] (0,-2)--(0,-1)--(1,-1)  (0,-4)--(1,-4)--(1,-2)--(2,-2)  (0,-6)--(2,-6)--(2,-3)--(3,-3) 
    (0,-8)--(0,-7)--(3,-7)--(3,-4)--(4,-4);
    \draw[line width=1pt, rounded corners] (0,0)--(0,-9)  (0,0)--(9,-9) ;
    \draw[line width=1pt, rounded corners] (0,0)--(0,-9)  (0,0)--(9,-9) ;
    \node[anchor=east] at (0,-1) {$\Tilde{A}_1$};
    \node[anchor=east] at (0,-2) {$\Tilde{A}_2$};
    \node[anchor=east] at (0,-3) {$\Tilde{A}_3$};
    \node[anchor=east] at (0,-4) {$\Tilde{A}_4$};
    \node[anchor=east] at (0,-5) {$\Tilde{A}_5$};
    \node[anchor=east] at (0,-6) {$\Tilde{A}_6$};
    \node[anchor=east] at (0,-7) {$\Tilde{A}_7$};
    \node[anchor=east] at (0,-8) {$\Tilde{A}_8$};
    \node[anchor=east] at (0,-9) {$\Tilde{A}_9$};
    \end{tikzpicture}
    \begin{tikzpicture}[scale=0.5]
        \draw[line width=2pt, rounded corners, color=orange](15.5,-8.5)--(16.5,-9.5)  
        (13.5,-6.5)--(14.5,-7.5)--(13.5,-8.5)--(14.5,-9.5)
        (11.5,-4.5)--(10.5,-5.5)--(13.5,-8.5)--(12.5,-9.5)
        (9.5,-2.5)--(6.5,-5.5)--(10.5,-9.5)
        (7.5,-0.5)--(4.5,-3.5)--(5.5,-4.5)--(4.5,-5.5)--(8.5,-9.5);
        \draw[line width=2pt, rounded corners, color=blue,dotted] (14.5,-7.5)--(15.5,-8.5)--(14.5,-9.5) 
        (12.5,-5.5)--(11.5,-6.5)--(13.5,-8.5)--(12.5,-9.5)
        (10.5,-3.5)--(8.5,-5.5)--(11.5,-8.5)--(10.5,-9.5)
        (8.5,-1.5)--(9.5,-2.5)--(6.5,-5.5)--(9.5,-8.5)--(8.5,-9.5);
          \node at (6.5,-0.5) {$1$};
          
          \node at (5.5,-1.5) {$1$};
          \node at (7.5,-1.5) {$2$};
          
          \node at (6.5,-2.5) {$2$};
          \node at (8.5,-2.5) {$2$};
          \node at (4.5,-2.5) {$1$};
          
          \node at (5.5,-3.5) {$2$};
          \node at (7.5,-3.5) {$2$};
          \node at (9.5,-3.5) {$4$};
          \node at (3.5,-3.5) {$1$};
          
          \node at (6.5,-4.5) {$2$};
          \node at (8.5,-4.5) {$4$};
          \node at (4.5,-4.5) {$1$};
          \node at (10.5,-4.5) {$5$};
          \node at (2.5,-4.5) {$1$};
          
          \node at (1.5,-5.5) {$1$};
          \node at (3.5,-5.5) {$1$};
          \node at (5.5,-5.5) {$2$};
          \node at (7.5,-5.5) {$4$};
          \node at (9.5,-5.5) {$5$};
          \node at (11.5,-5.5) {$6$};
          
          \node at (0.5,-6.5) {$1$};
          \node at (2.5,-6.5) {$1$};
          \node at (4.5,-6.5) {$1$};
          \node at (6.5,-6.5) {$2$};
          \node at (8.5,-6.5) {$4$};
          \node at (10.5,-6.5) {$5$};
          \node at (12.5,-6.5) {$7$};
          
          \node at (-0.5,-7.5) {$1$};
          \node at (1.5,-7.5) {$1$};
          \node at (3.5,-7.5) {$1$};
          \node at (5.5,-7.5) {$1$};
          \node at (7.5,-7.5) {$2$};
          \node at (9.5,-7.5) {$4$};
          \node at (11.5,-7.5) {$5$};
          \node at (13.5,-7.5) {$7$};
          
          \node at (-1.5,-8.5) {$1$};
          \node at (0.5,-8.5) {$1$};
          \node at (2.5,-8.5) {$1$};
          \node at (4.5,-8.5) {$1$};
          \node at (6.5,-8.5) {$1$};
          \node at (8.5,-8.5) {$2$};
          \node at (10.5,-8.5) {$4$};
          \node at (12.5,-8.5) {$5$};
          \node at (14.5,-8.5) {$8$};

          \node at (-2.5,-9.5) {$1$};
          \node at (-0.5,-9.5) {$1$};
          \node at (1.5,-9.5) {$1$};
          \node at (3.5,-9.5) {$1$};
          \node at (5.5,-9.5) {$1$};
          \node at (7.5,-9.5) {$1$};
          \node at (9.5,-9.5) {$3$};
          \node at (11.5,-9.5) {$5$};
          \node at (13.5,-9.5) {$7$};
          \node at (15.5,-9.5) {$9$};
          
          \node at (-1.5,-11.5) {$18$};
          \node at (0.5,-11.5) {$16$};
          \node at (2.5,-11.5) {$14$};
          \node at (4.5,-11.5) {$12$};
          \node at (6.5,-11.5) {$10$};
          \node at (8.5,-11.5) {$8$};
          \node at (10.5,-11.5) {$6$};
          \node at (12.5,-11.5) {$4$};
          \node at (14.5,-11.5) {$2$};
          \node at (-2.5,-11.5) {$19$};
          \node at (-0.5,-11.5) {$17$};
          \node at (1.5,-11.5) {$15$};
          \node at (3.5,-11.5) {$13$};
          \node at (5.5,-11.5) {$11$};
          \node at (7.5,-11.5) {$9$};
          \node at (9.5,-11.5) {$7$};
          \node at (11.5,-11.5) {$5$};
          \node at (13.5,-11.5) {$3$};
          \node at (15.5,-11.5) {$1$};
          
          \node[color=blue] at (8.5,-0.5) {$\Hat{A}_{9}$};
          \node[color=blue] at (9.5,-1.5) {$\Hat{A}_{8}$};
          \node[color=blue] at (10.5,-2.5) {$\Hat{A}_{7}$};
          \node[color=blue] at (11.5,-3.5) {$\Hat{A}_{6}$};
          \node[color=blue] at (12.5,-4.5) {$\Hat{A}_{5}$};
          \node[color=blue] at (13.5,-5.5) {$\Hat{A}_{4}$};
          \node[color=blue] at (14.5,-6.5) {$\Hat{A}_{3}$};
          \node[color=blue] at (15.5,-7.5) {$\Hat{A}_{2}$};
          \node[color=blue] at (16.5,-8.5) {$\Hat{A}_{1}$};
          \fill[color=blue](7.5,-0.5) circle (3pt);
          \fill[color=blue](8.5,-1.5) circle (3pt);
          \fill[color=blue](9.5,-2.5) circle (3pt);
          \fill[color=blue](10.5,-3.5) circle (3pt);
          \fill[color=blue](11.5,-4.5) circle (3pt);
          \fill[color=blue](12.5,-5.5) circle (3pt);
          \fill[color=blue](13.5,-6.5) circle (3pt);
          \fill[color=blue](14.5,-7.5) circle (3pt);
          \fill[color=blue](15.5,-8.5) circle (3pt);
          
          \node[color=red] at (16.5,-10.5) {$\Hat{B}_{0}$};
          \node[color=red] at (14.5,-10.5) {$\Hat{B}_{1}$};
          \node[color=red] at (12.5,-10.5) {$\Hat{B}_{2}$};
          \node[color=red] at (10.5,-10.5) {$\Hat{B}_{3}$};
          \node[color=red] at (8.5,-10.5) {$\Hat{B}_{4}$};
          \node[color=red] at (6.5,-10.5) {$\Hat{B}_{5}$};
          \node[color=red] at (4.5,-10.5) {$\Hat{B}_{6}$};
          \node[color=red] at (2.5,-10.5) {$\Hat{B}_{7}$};
          \node[color=red] at (0.5,-10.5) {$\Hat{B}_{8}$};
          \node[color=red] at (-1.5,-10.5) {$\Hat{B}_{9}$};
          \fill[color=red](0.5,-9.5) circle (3pt);
          \fill[color=red](2.5,-9.5) circle (3pt);
          \fill[color=red](4.5,-9.5) circle (3pt);
          \fill[color=red](6.5,-9.5) circle (3pt);
          \fill[color=red](8.5,-9.5) circle (3pt);
          \fill[color=red](10.5,-9.5) circle (3pt);
          \fill[color=red](12.5,-9.5) circle (3pt);
          \fill[color=red](14.5,-9.5) circle (3pt);
          \fill[color=red](16.5,-9.5) circle (3pt);
          \fill[color=red](-1.5,-9.5) circle (3pt);
    \end{tikzpicture}
    \centering
    \label{Figure4}
    \caption{The tuple of paths $\Tilde{T}\in\Tilde{\mathcal{P}}_{10}^{1,12}$ from \hyperref[Figure2]{Figure 2} and its corresponding Gelfand-Tsetlin-pattern.}
    \centering
    \end{figure}

Now we have to translate our conditions on the tuples $\Tilde{T}\in\Tilde{\mathcal{P}}_n^{l,r}$. Recall that $\Tilde{P}_j$ is the set of all $(n-1)$-tuples of kissing lattice paths $\Tilde{P}_j$ consisting of north- and east-steps in the integer lattice with
\begin{itemize}
    \item starting points $\Tilde{A}_1=(0,-1),...,\Tilde{A}_{n-1}=(0,-n+1)$,
    \item end-points in $\Tilde{B}_0=(0,0),...,\Tilde{B}_{r-n+2}(r-n+2,n-r-2)$,
    \item the path $\Tilde{P}_{n-1}$ starting at $\Tilde{S}_{n-1}$ staying weakly above the line $y=x+l-r-1$,
\end{itemize}
where a path is weighted by $t$ if and only if it ends with a north-step. 

The second condition simply translates into the fact that the paths in $G(\Tilde{T})$ end at the points $\Hat{B}_0,...,\Hat{B}_{r+1-(n-1)}$. This means that to the left of the SE-NW diagonal through the point $\Hat{B}_{r+2-n}$ all entries of $G(\Tilde{T})$ are equal to $1$. In other words all entries to the left and on the $(r+3-n)$-th north-west diagonal of the Gelfand-Tsetlin-pattern are equal to $1$.

Note that the intersection point of the line $y=-x$ and $y=x+l-r-1$ is given by $(\frac{r+1-l}{2},-\frac{r+1-l}{2})$ and because south-west to north-east steps in $\Tilde{T}$ translate to simple south steps in $G(\Tilde{T})$, the fact that $\Tilde{P}_{n-1}$ stays weakly above the line $y=x+l-r-1$, translates into the fact that the path $\Hat{P}_{n-1}$ stays weakly to the right of the $(r+1-l)$-th column of the Gelfand-Tsetlin-pattern. See for example the tuple of paths $\Tilde{T}\in\Tilde{\mathcal{P}}_{10}^{1,12}$ from \hyperref[Figure4]{Figure 4} and its corresponding Gelfand-Tsetlin-pattern with the $(r+1-l)$-th column marked in \hyperref[Figure5]{Figure 5}.
\begin{figure}[ht]
   \begin{tikzpicture}[scale=0.5]
    \draw[step=1 cm, color=gray!50, very thin] (0, 0) grid (9,-9);
    \draw[line width=3pt, rounded corners, color=orange](0,-1)--(0,0)  (0,-3)--(0,-2)--(1,-2)--(1,-1)  (0,-5)--(1,-5)--(1,-4)--(1,-2)--(2,-2)  (0,-7)--(3,-7)--(3,-3)
    (0,-9)--(3,-9)--(3,-8)--(4,-8)--(4,-4);
    \draw[line width=3pt, rounded corners, color=blue, dotted] (0,-2)--(0,-1)--(1,-1)  (0,-4)--(1,-4)--(1,-2)--(2,-2)  (0,-6)--(2,-6)--(2,-3)--(3,-3) 
    (0,-8)--(0,-7)--(3,-7)--(3,-4)--(4,-4);
    \draw[line width=1pt, rounded corners] (0,0)--(0,-9)  (0,0)--(9,-9) ;
    \draw[line width=1pt, rounded corners, color=red] (4,-9)--(9,-4);
    \draw[line width=1pt, rounded corners, color=green] (3,-9)--(9,-3);
    \node[anchor=west, color=red] at (9,-4) {$y=x+l-r-2$};
    \node[anchor=west, color=green] at (9,-3) {$y=x+l-r-1$};
    \node[anchor=east] at (0,-1) {$\Tilde{A}_1$};
    \node[anchor=east] at (0,-2) {$\Tilde{A}_2$};
    \node[anchor=east] at (0,-3) {$\Tilde{A}_3$};
    \node[anchor=east] at (0,-4) {$\Tilde{A}_4$};
    \node[anchor=east] at (0,-5) {$\Tilde{A}_5$};
    \node[anchor=east] at (0,-6) {$\Tilde{A}_6$};
    \node[anchor=east] at (0,-7) {$\Tilde{A}_7$};
    \node[anchor=east] at (0,-8) {$\Tilde{A}_8$};
    \node[anchor=east] at (0,-9) {$\Tilde{A}_9$};
    \end{tikzpicture}
    \begin{tikzpicture}[scale=0.5]
        \draw[line width=2pt, rounded corners, color=orange](15.5,-8.5)--(16.5,-9.5)  
        (13.5,-6.5)--(14.5,-7.5)--(13.5,-8.5)--(14.5,-9.5)
        (11.5,-4.5)--(10.5,-5.5)--(13.5,-8.5)--(12.5,-9.5)
        (9.5,-2.5)--(6.5,-5.5)--(10.5,-9.5)
        (7.5,-0.5)--(4.5,-3.5)--(5.5,-4.5)--(4.5,-5.5)--(8.5,-9.5);
        \draw[line width=2pt, rounded corners, color=blue,dotted] (14.5,-7.5)--(15.5,-8.5)--(14.5,-9.5) 
        (12.5,-5.5)--(11.5,-6.5)--(13.5,-8.5)--(12.5,-9.5)
        (10.5,-3.5)--(8.5,-5.5)--(11.5,-8.5)--(10.5,-9.5)
        (8.5,-1.5)--(9.5,-2.5)--(6.5,-5.5)--(9.5,-8.5)--(8.5,-9.5);
        \draw[line width=1pt, rounded corners, color=green] (4,-0)--(5,0)--(5,-12)--(4,-12)--cycle ;
          \node at (6.5,-0.5) {$1$};
          
          \node at (5.5,-1.5) {$1$};
          \node at (7.5,-1.5) {$2$};
          
          \node at (6.5,-2.5) {$2$};
          \node at (8.5,-2.5) {$2$};
          \node at (4.5,-2.5) {$1$};
          
          \node at (5.5,-3.5) {$2$};
          \node at (7.5,-3.5) {$2$};
          \node at (9.5,-3.5) {$4$};
          \node at (3.5,-3.5) {$1$};
          
          \node at (6.5,-4.5) {$2$};
          \node at (8.5,-4.5) {$4$};
          \node at (4.5,-4.5) {$1$};
          \node at (10.5,-4.5) {$5$};
          \node at (2.5,-4.5) {$1$};
          
          \node at (1.5,-5.5) {$1$};
          \node at (3.5,-5.5) {$1$};
          \node at (5.5,-5.5) {$2$};
          \node at (7.5,-5.5) {$4$};
          \node at (9.5,-5.5) {$5$};
          \node at (11.5,-5.5) {$6$};
          
          \node at (0.5,-6.5) {$1$};
          \node at (2.5,-6.5) {$1$};
          \node at (4.5,-6.5) {$1$};
          \node at (6.5,-6.5) {$2$};
          \node at (8.5,-6.5) {$4$};
          \node at (10.5,-6.5) {$5$};
          \node at (12.5,-6.5) {$7$};
          
          \node at (-0.5,-7.5) {$1$};
          \node at (1.5,-7.5) {$1$};
          \node at (3.5,-7.5) {$1$};
          \node at (5.5,-7.5) {$1$};
          \node at (7.5,-7.5) {$2$};
          \node at (9.5,-7.5) {$4$};
          \node at (11.5,-7.5) {$5$};
          \node at (13.5,-7.5) {$7$};
          
          \node at (-1.5,-8.5) {$1$};
          \node at (0.5,-8.5) {$1$};
          \node at (2.5,-8.5) {$1$};
          \node at (4.5,-8.5) {$1$};
          \node at (6.5,-8.5) {$1$};
          \node at (8.5,-8.5) {$2$};
          \node at (10.5,-8.5) {$4$};
          \node at (12.5,-8.5) {$5$};
          \node at (14.5,-8.5) {$8$};

          \node at (-2.5,-9.5) {$1$};
          \node at (-0.5,-9.5) {$1$};
          \node at (1.5,-9.5) {$1$};
          \node at (3.5,-9.5) {$1$};
          \node at (5.5,-9.5) {$1$};
          \node at (7.5,-9.5) {$1$};
          \node at (9.5,-9.5) {$3$};
          \node at (11.5,-9.5) {$5$};
          \node at (13.5,-9.5) {$7$};
          \node at (15.5,-9.5) {$9$};
          
          \node at (-1.5,-11.5) {$18$};
          \node at (0.5,-11.5) {$16$};
          \node at (2.5,-11.5) {$14$};
          \node at (4.5,-11.5) {$12$};
          \node at (6.5,-11.5) {$10$};
          \node at (8.5,-11.5) {$8$};
          \node at (10.5,-11.5) {$6$};
          \node at (12.5,-11.5) {$4$};
          \node at (14.5,-11.5) {$2$};
          \node at (-2.5,-11.5) {$19$};
          \node at (-0.5,-11.5) {$17$};
          \node at (1.5,-11.5) {$15$};
          \node at (3.5,-11.5) {$13$};
          \node at (5.5,-11.5) {$11$};
          \node at (7.5,-11.5) {$9$};
          \node at (9.5,-11.5) {$7$};
          \node at (11.5,-11.5) {$5$};
          \node at (13.5,-11.5) {$3$};
          \node at (15.5,-11.5) {$1$};
          
          \node[color=blue] at (8.5,-0.5) {$\Hat{A}_{9}$};
          \node[color=blue] at (9.5,-1.5) {$\Hat{A}_{8}$};
          \node[color=blue] at (10.5,-2.5) {$\Hat{A}_{7}$};
          \node[color=blue] at (11.5,-3.5) {$\Hat{A}_{6}$};
          \node[color=blue] at (12.5,-4.5) {$\Hat{A}_{5}$};
          \node[color=blue] at (13.5,-5.5) {$\Hat{A}_{4}$};
          \node[color=blue] at (14.5,-6.5) {$\Hat{A}_{3}$};
          \node[color=blue] at (15.5,-7.5) {$\Hat{A}_{2}$};
          \node[color=blue] at (16.5,-8.5) {$\Hat{A}_{1}$};
          \fill[color=blue](7.5,-0.5) circle (3pt);
          \fill[color=blue](8.5,-1.5) circle (3pt);
          \fill[color=blue](9.5,-2.5) circle (3pt);
          \fill[color=blue](10.5,-3.5) circle (3pt);
          \fill[color=blue](11.5,-4.5) circle (3pt);
          \fill[color=blue](12.5,-5.5) circle (3pt);
          \fill[color=blue](13.5,-6.5) circle (3pt);
          \fill[color=blue](14.5,-7.5) circle (3pt);
          \fill[color=blue](15.5,-8.5) circle (3pt);
          
          \node[color=red] at (16.5,-10.5) {$\Hat{B}_{0}$};
          \node[color=red] at (14.5,-10.5) {$\Hat{B}_{1}$};
          \node[color=red] at (12.5,-10.5) {$\Hat{B}_{2}$};
          \node[color=red] at (10.5,-10.5) {$\Hat{B}_{3}$};
          \node[color=red] at (8.5,-10.5) {$\Hat{B}_{4}$};
          \node[color=red] at (6.5,-10.5) {$\Hat{B}_{5}$};
          \node[color=red] at (4.5,-10.5) {$\Hat{B}_{6}$};
          \node[color=red] at (2.5,-10.5) {$\Hat{B}_{7}$};
          \node[color=red] at (0.5,-10.5) {$\Hat{b}_{8}$};
          \node[color=red] at (-1.5,-10.5) {$\Hat{B}_{9}$};
          \fill[color=red](0.5,-9.5) circle (3pt);
          \fill[color=red](2.5,-9.5) circle (3pt);
          \fill[color=red](4.5,-9.5) circle (3pt);
          \fill[color=red](6.5,-9.5) circle (3pt);
          \fill[color=red](8.5,-9.5) circle (3pt);
          \fill[color=red](10.5,-9.5) circle (3pt);
          \fill[color=red](12.5,-9.5) circle (3pt);
          \fill[color=red](14.5,-9.5) circle (3pt);
          \fill[color=red](16.5,-9.5) circle (3pt);
          \fill[color=red](-1.5,-9.5) circle (3pt);
    \end{tikzpicture}
    \centering
    \label{Figure5}
    \caption{The tuple of paths $\Tilde{T}\in\Tilde{\mathcal{P}}_{10}^{1,12}$ from \hyperref[Figure2]{Figure 2} and its corresponding Gelfand-Tsetlin-pattern with the $(r+1-l)$-th column marked.}
    \centering
    \end{figure}

Since a path is weighted by $t$ in $\Tilde{T}$ if and only if it ends with a north step, the Gelfand-Tsetlin pattern $G(\Tilde{T})$ is weighted by $$t^{n-1}\prod_{i=1}^{r+2-n}t^{a_{n-1,i}-a_{n,i}}.$$ In order to see this, first observe that $\prod_{i=1}^{r+2-n}t^{a_{n,i}-a_{n-1,i}}$ counts the paths in $G(\Tilde{T})$ that end with a south-west step, because in any path ending with a south-west step one creates a difference of $1$ between the entry in the bottom column to the right of the end-point of the path and the entry directly above it along the SE-NW diagonal, when filling in the numbers between the paths as in the construction. Since there are $n-1$ paths, the statement follows.  

We see that the set $\mathcal{P}_n^{l,r}$ is in weight-preserving bijection with the set $\mathcal{G}_n^{l,r}$ of all Gelfand-Tsetlin-patterns $G$ of order $n$ where
\begin{itemize}
    \item all entries to the left and on the $(r+3-n)$-th SE-NW diagonal are equal to $1$,
    \item all entries to the left and on the $(r+1-l)$-th column are equal to $1$,
\end{itemize}
and $G$ is weighted by $t^{n-1}\prod_{i=1}^{r+2-n}t^{a_{n-1,i}-a_{n,i}}$. 

Note that because $n-1\leq r\leq 2n-3$, we have $2\leq r+3-n\leq n$ and since $0\leq l\leq n-2$ we have $2\leq r+1-l\leq 2n-2$, hence both $r+3-n$ and $r+1-l$ are the index of an existing diagonal and column. 

Now by simply rotating by $45^\circ$ and cutting off all the entries forced to be equal to $1$, we obtain a weight preserving bijection between $\mathcal{P}_n$ and the set of all $(n,r+2-n,r-l)$-Magog pentagons where a pentagon is weighted by $t^{n-1}\prod_{i=1}^{r+2-n}t^{a_{n-1,i}-a_{n,i}}$. Since in Theorem \ref{SumOfDeterminantsLRRestrictedASTFormula} the number of $(n,l,r)$-ASPs $P$ with $\rho(P)=p$ is given by the coefficient of $t^{p-1}$, we multiply the statistic on the Magog pentagons with an additional $t$, giving us the $\tau$-statistic and the statement of the theorem. For the example $T\in\mathcal{P}_{10}^{1,12}$ above, we would obtain the $(10,4,11)$-Magog pentagon $M$ below with weight $\tau(M)-1=4=\#\textrm{paths ending with a north step in the example }T$.
$$\begin{matrix}
1&2&2&4&5&6&7&7&8&9\\
1&2&2&4&5&5&5&5&7\\
&2&2&4&4&4&4&5\\
& &2&2&2&2&3\\
\end{matrix}$$
\end{proof}
\section{Proof of Theorem \ref{PfaffianCor}}\label{Section Proof of Pfaffian}
Again one of the main ingredients of the proof of Theorem \ref{PfaffianCor} will be Theorem \ref{SumOfDeterminantsLRRestrictedASTFormula} that was established in Section \ref{SectionIntermediateResult}. The other ingredient is a result of Stembridge, which we will state in a form useful for our purposes. For this we define the \emph{Pfaffian} of a $2n\times 2n$-skew symmetric matrix. 
\begin{defin}
    Let $A=(a_{i,j})_{1\leq i,j\leq 2n}$ be a skew symmetric matrix. We define its \emph{Pfaffian} by 
    \begin{equation}
        \underset{1\leq i<j\leq 2n}{\mathrm{Pf}}(a_{i,j})\defeq\frac{1}{2^n n!}\sum_{\sigma\in S_{2n}}\Sign(\sigma)\prod_{i=1}^na_{\sigma(2i-1),\sigma(2i)},
    \end{equation}
    where $S_{2n}$ is the \emph{symmetric group} of order $(2n)!$ and $\Sign$ is the \emph{sign} of a permutation. 
\end{defin}
\begin{thm}[\cite{NonintersectingPathsPfaffiansAndPlanePartitions}, Theorem 3.1]
Let $G$ be an acyclic, directed graph, $u_1,...,u_r$ be an $r$-tuple of vertices of $G$, $I$ a totally ordered subset of the vertices of $G$ such that for all $i<j\in [r]$ and $v<v'\in I$ any directed path in $G$ from $u_j$ to $v$ intersects any directed path in $G$ from $u_i$ to $v'$. Denote by $$\mathcal{GF}_0[u_1,...,u_r;I\mid G]=\sum_{(P_1,...,P_r)}\prod_{i=1}^r\wt(P_i)$$
where we sum over all $r$-tuples of non-intersecting directed paths with starting points $u_1,...,u_r$ and end-points in $I$ inside $G$ and $\wt(P_i)$ is the weight of the path $P_i$.

If $r$ is even we have that
\begin{equation*}
    \mathcal{GF}_0[u_1,...,u_r;I\mid G]=\underset{1\leq i<j\leq r}{\mathrm{Pf}}\Big(\mathcal{GF}_0[u_i,u_j;I\mid G]\Big).
\end{equation*}
\end{thm}

Note that if $r$ is odd we may adjoin a phantom vertex $u_{r+1}$ to $G$ with no incident edges. We then include $u_{r+1}$ in $I$ and order all vertices of $I$ before it. Denote $I\cup \{u_{r+1}\}$ by $I^*$. Now $u_1,...,u_r,u_{r+1}$ and the new $I^*$ satisfy all conditions of the above statement and we can provide a Pfaffian of order $r+1$ for $\mathcal{GF}_0[u_1,...,u_r;I]=\mathcal{GF}_0[u_1,...,u_r,u_{r+1};I^*]$. Note that in this case for all $i\in[r]$ we have $\mathcal{GF}_0[u_i,u_{r+1};I]=\mathcal{GF}_0[u_i;I]$. \\ 

Fix $n\in\mathbb{N}$ and $0\leq l\leq n-2<r\leq 2n-3$ such that $l+r-2n+2<0$ and $\frac{2n+r-l}{3}>n-1$. Now consider the graph $G_{n,l,r}$ with vertex set consisting of all integer lattice points in the plane lying weakly above the line $y=x+l-r-2n+1$ and oriented edges given by north- and east-steps between those points. This graph is directed and clearly acyclic. 

Let $A_i\defeq(-i,-2i)$ as in Proposition \ref{PropRelatingSumOfDetsAndLatticePaths} and consider $I\defeq\{(1,-1),...,(r+1,-r-1)\}$. It is clear by elementary geometric arguments, see \hyperref[Figure1]{Figure 1}, that these vertices satisfy the conditions of Stembridge's Theorem. We assign a weight $t$ to a path in $G_{n,l,r}$ if and only if it ends with a north step. By Proposition \ref{PropRelatingSumOfDetsAndLatticePaths}, we know that 
\begin{multline*}
         \mathcal{GF}_0[A_1,...,A_{n-1};I\mid G_{n,l,r}]=\sum_{1\leq e_1<e_2<...<e_{n-1}\leq r+1}\det\bigg(t\binom{j-1}{e_i-j}\\
         -t\binom{j-1}{r-e_i-l+2n-1-j}
        +\binom{j-1}{e_i-1-j}-\binom{j-1}{r-e_i-l+2n-j}\bigg)_{i,j=1}^{n-1}.
\end{multline*}
Furthermore we can, similarly as in the proof of Proposition \ref{PropRelatingSumOfDetsAndLatticePaths}, deduce that for $i,j\in[n-1]$ we have
\begin{multline*}
    \mathcal{GF}_0[A_i,A_{j};I\mid G_{n,l,r}]\\
    = \sum_{1\leq e_1<e_2\leq r+1}\det\begin{pmatrix}t\binom{j-1}{e_c-j}-t\binom{j-1}{r-e_c-l+2n-1-j}+\binom{j-1}{e_c-1-j}-\binom{j-1}{r-e_c-l+2n-j}\\t\binom{i-1}{e_c-i}-t\binom{i-1}{r-e_c-l+2n-1-i}+\binom{i-1}{e_c-1-i}-\binom{i-1}{r-e_c-l+2n-i}\end{pmatrix}_{c=1}^2,
\end{multline*}
where we denote by $\det(v_i)_{i=1}^n$ the determinant of the matrix with column vectors $v_1,..,v_n$. 
Note that for $j\in[n-1]$ we have
\begin{multline*}
    \mathcal{GF}_0[A_{j};I\mid G_{n,l,r}]=\sum_{1\leq e_1\leq r+1}\Bigg( t\binom{j-1}{e_1-j}
    -t\binom{j-1}{r-e_1-l+2n-1-j}\\
    +\binom{j-1}{e_1-1-j}-\binom{j-1}{r-e_1-l+2n-j}\Bigg),
\end{multline*}
by a similar reasoning. Therefore by Stembridge's Theorem we can deduce that for odd $n$ we have that
\begin{multline*}
         \sum_{1\leq e_1<e_2<...<e_{n-1}\leq r+1}\det\bigg(t\binom{j-1}{e_i-j}-t\binom{j-1}{r-e_i-l+2n-1-j}
        +\binom{j-1}{e_i-1-j}\\
        -\binom{j-1}{r-e_i-l+2n-j}\bigg)_{i,j=1}^{n-1}\\
        =\mathcal{GF}_0[A_1,...,A_{n-1};I\mid G_{n,l,r}]
        =\underset{1\leq i<j\leq n-1}{\mathrm{Pf}}\Big(\mathcal{GF}_0[A_i,A_{j};I\mid G_{n,l,r}]\big)\\
        =\underset{1\leq i<j\leq n-1}{\mathrm{Pf}}\Bigg(\sum_{1\leq e_1<e_2\leq r+1}\det\begin{pmatrix}t\binom{j-1}{e_c-j}-t\binom{j-1}{r-e_c-l+2n-1-j}+\binom{j-1}{e_c-1-j}-\binom{j-1}{r-e_c-l+2n-j}\\t\binom{i-1}{e_c-i}-t\binom{i-1}{r-e_c-l+2n-1-i}+\binom{i-1}{e_c-1-i}-\binom{i-1}{r-e_c-l+2n-i}\end{pmatrix}_{c=1}^2\Bigg)\\
        =\pfaf{n-1}\Bigg(\sum_{1\leq e_1<e_2\leq r+1}\det\Big(\scalebox{0.9}{$t\binom{k_j-1}{e_i-k_j}-t\binom{k_j-1}{r-e_i-l+2n-1-k_j}+\binom{k_j-1}{e_i-1-k_j}-\binom{k_j-1}{r-e_i-l+2n-k_j}$}\Big)_{i,j=1}^2\Bigg).
\end{multline*}
The above equality still holds if $n$ is even, but by Stembridge's Theorem we have to append an $n$-th column to the Pfaffian with entries 
\begin{multline*}
    a_{j,n}\defeq \mathcal{GF}_0[A_{j};I\mid G_{n,l,r}]=
    \sum_{1\leq e_1\leq r+1}\Bigg( t\binom{j-1}{e_1-j}
    -t\binom{j-1}{r-e_1-l+2n-1-j}
    \\+\binom{j-1}{e_1-1-j}-\binom{j-1}{r-e_1-l+2n-j}\Bigg).
\end{multline*}
Now Theorem \ref{PfaffianCor} follows from Theorem \ref{SumOfDeterminantsLRRestrictedASTFormula} and Theorem \ref{RelationASTsMagogPentagons}.
\section{Consequences, conjectures and connections to Gog trapezoids, Gog pentagons and GOGAm pentagons}\label{Section Consequences and conjectures}
As mentioned in Section \ref{SectionPrerequisitesAndMainResults}, the definition of $(n,l,r)$-ASPs is inspired by the conjectures of Behrend in \cite{FischerEnumerationOfAlternatingSignTrianglesUsingAConstantTermApproach}. The main goal of the paper was to relate the AST side to the Magog-side, which we managed to accomplish within Theorem \ref{RelationASTsMagogPentagons}. This gives us a new point of view on some open conjectures connected to ASPs and \textit{Gog pentagons} as well as Magog pentagons. In this chapter we will discuss how exactly our results interact or extend those aforementioned conjectures and while the full conjecture of Behrend remains open, how we got a part of it as a by-product of our investigations.

To formulate the precise results we obtained, we will need the following three statistics for an ASM $M$ of order $n$.
\begin{itemize}
    \item $T_L(M)$ the number of NE-SW diagonals, counted starting in the top left corner, consisting of all zeros,
    \item $T_R(M)$ the number of NW-SE diagonals, counted starting in the top right corner, consisting of all zeros,
    \item $\rho(M)$ the position of the unique $1$ in the top row.
\end{itemize}

Let us also formally introduce Gog trapezoids. 
\begin{defin}
    An $(m,n,k)$-\emph{Gog trapezoid} is an array of positive integers consisting of the first $k$ columns of an array
    $$\begin{matrix}
        a_{1,1}&a_{1,2}&...&...&a_{1,n}\\
        a_{2,1}&a_{2,2}&...&a_{2,n-1}&\\
        ...&...&...& &\\
        a_{n-1,1}&a_{n-1,2}& & &\\
        a_{n,1}& & & &
    \end{matrix}$$
    such that entries along rows are strictly increasing, entries along columns are weakly increasing, and entries along diagonals from lower-left to upper-right are weakly increasing, and such that the entries in the right most column are bounded by $a_{1,k}\leq m+k, a_{2,k}\leq m+k+1,...,a_{n+1-k,k}\leq m+n$. We will suppress the $m$ whenever it is zero and simply write $(n,k)$-Gog trapezoid for any $(0,n,k)$-Gog trapezoid.
\end{defin}
It is not hard to see that we can transform any $n\times n$ - ASM one to one into a $(n,n)$-Gog trapezoid by first mapping the ASM to the matrix obtained by replacing the $i$-th row of the ASM by the sum of rows $i,i+1,..,n$ and then recording the positions of the $1$'s in the so obtained matrix in a triangular array. 
\begin{equation*}
    \begin{pmatrix}
        0&0&1&0&0&0\\
        1&0&-1&1&0&0\\
        0&0&1&-1&0&1\\
        0&1&-1&1&0&0\\
        0&0&1&-1&1&0\\
        0&0&0&1&0&0
    \end{pmatrix}\quad\rightarrow\quad
    \begin{pmatrix}
        1&1&1&1&1&1\\
        1&1&0&1&1&1\\
        0&1&1&0&1&1\\
        0&1&0&1&1&0\\
        0&0&1&0&1&0\\
        0&0&0&1&0&0
    \end{pmatrix}\quad\rightarrow\quad
    \begin{matrix}
        1&2&3&4&5&6\\
        1&2&4&5&6\\
        2&3&5&6\\
        2&4&5\\
        3&5\\
        4
    \end{matrix}
\end{equation*}
Note that under this bijection the statistics $T_L(M)$ and $T_R(M)$ of an ASM $M$ translate into the number of all $1$ NE-SW diagonals, counted starting in the top left corner, and the number of all $1$ NW-SE diagonals, counted starting in the top right corner, of the in-between matrix respectively. These then translate into fixed pattern triangles of the Gog trapezoid. More precisely, if $T_L(M)=l$ then we have the pattern 
\begin{equation*}
    \begin{matrix}
    1&2&...&l+1\\
    1&...&l&\\
    \vdots&\ddots\\
    1
    \end{matrix}
\end{equation*} 
in the top left corner of the Gog trapezoid and if $T_R(M)=r$ then we have the pattern
\begin{equation*}
    \begin{matrix}
    n-r&n+1-r&...&n\\
    n+1-r&...&n&\\
    \vdots&\ddots\\
    n
    \end{matrix}
\end{equation*}
in the top right corner of the Gog trapezoid. In particular this tells us that the number of $n\times n$ - ASM with $T_R(M)=r$ is equal to the number of $(n,n-1-r)$-Gog trapezoids. 
The following statement is a simple corollary of Theorem \ref{RelationASTsMagogPentagons}.
\begin{cor}\label{ASTsAndMagogTrapezoids}
Let $3<n\in\mathbb{N}$, $1\leq p\leq n$ and $n-1\leq r\leq 2n-3$. The number of $(n,0,r)$-ASPs $T$ of order $n$ with $\rho(T)=p$ is equal to the number of $(n,2n-3-r,2n-3)$-ASPs $T$ of order $n$ with $\rho(T)=n+1-p$ and equal to the number of $(n,r+2-n)$-Magog trapezoids $P$ with $\tau(P)=p$.
\end{cor}
\begin{proof}
Note that a $(n,r+2-n,r)$-Magog pentagon is simply a $(n,r+2-n)$-Magog trapezoid, since the $r$-th NW-SE diagonal of the pentagon lies to the left of the $(r+2-n)$-th row. The claim now follows directly from Theorem \ref{RelationASTsMagogPentagons}.
\end{proof}
Combining the above with Zeilberger's famous Lemma 1 we obtain a partial proof of Behrend's conjecture. 
\begin{cor}
Let $n\in\mathbb{N}$ and $n-1\leq r\leq2n-3$. The number of ASMs $M$ of order $n$ with $T_R(M)=2n-3-r$ is equal to the number of $(n,0,r)$-ASPs and the number of ASMs $M$ of order $n$ with $T_L(M)=2n-3-r$ is equal to the number of $(n,2n-3-r,2n-3)$-ASPs. 
\end{cor}
\begin{proof}
By Corollary \ref{ASTsAndMagogTrapezoids} we know that the number of $(n,0,r)$-ASPs of order $n$ is equal to the number of $(n,r+2-n)$-Magog trapezoids. In \cite{ProofOfTheAlternatingSignMatrixConjecture} Zeilberger proved in his Lemma 1, that the number of such $(n,r+2-n)$-Magog trapezoids is equal to the number of $(n,r+2-n)$-Gog trapezoids. By the above observation we know that these $(n,r+2-n)$-Gog trapezoids are equinumerous with $n\times n$ - ASMs $M$ where $T_R(M)=n-1-(r+2-n)=2n-3-r$. The second statement follows by the horizontal reflection symmetry of both classes of objects.
\end{proof}
\begin{exam} The five $(3,0,2)$-ASPs and five $3\times 3$ ASMs with $T_R=1$.
\begin{gather*}
    \begin{matrix}1&0&0&0&0\\&1&0&0&\\&&1&&\end{matrix}\quad\textrm{,}\quad
    \begin{matrix}1&0&0&0&0\\&0&0&1&\\&&1&&\end{matrix}\quad\textrm{,}\quad
    \begin{matrix}0&0&0&1&0\\&1&0&0&\\&&1&&\end{matrix}\quad\textrm{,}\quad\\
    \begin{matrix}0&1&0&0&0\\&0&0&1&\\&&1&&\end{matrix}\quad\textrm{,}\quad
    \begin{matrix}0&0&1&0&0\\&1&-1&1&\\&&1&&\end{matrix}\quad\textrm{,}\quad\\
    \begin{pmatrix}
        0&1&0\\
        1&-1&1\\
        0&1&0
    \end{pmatrix}\quad ,\quad  
    \begin{pmatrix}
        0&1&0\\
        1&0&0\\
        0&0&1
    \end{pmatrix}
    \quad ,\quad  
    \begin{pmatrix}
        0&1&0\\
        0&0&1\\
        1&0&0
    \end{pmatrix}
    \quad ,\quad  \\
    \begin{pmatrix}
        1&0&0\\
        0&1&0\\
        0&0&1
    \end{pmatrix}
    \quad ,\quad  
    \begin{pmatrix}
        1&0&0\\
        0&0&1\\
        0&1&0
    \end{pmatrix}
\end{gather*}
\end{exam}
Let us introduce another object.
\begin{defin}
    An $(m,n,k,l)$-Gog pentagon is an $(m,n,k)$-Gog trapezoid with the top left corner being equal to the fixed pattern
    \begin{equation*}
    \begin{matrix}
    1&2&...&l\\
    1&...&l-1&\\
    \vdots&\vdots\\
    1
    \end{matrix}
\end{equation*} 
Since this pattern is fixed we may cut it off. This justifies the name pentagon. We will suppress the $m$ whenever it is zero and simply write $(n,k,l)$-Gog pentagon for any $(0,n,k,l)$-Gog pentagon. See for example below the $(6,4,2)$-Gog pentagon.
\begin{equation*}
    \begin{matrix}
        1&2&3&4\\
        1&3&4&5\\
        2&3&4&5\\
        2&4&5\\
        5&6\\
        6
    \end{matrix}\quad\leftrightarrow\quad
    \begin{matrix}
        &&3&4\\
        &3&4&5\\
        2&3&4&5\\
        2&4&5\\
        5&6\\
        6
    \end{matrix}
\end{equation*}
\end{defin}
With this new object we can describe the translation of all three statistics $T_L(M)$ and $T_R(M)$ and the position of the unique $1$ in the top row of an ASM $M$ to the Gog trapezoid side. That is, if $T_L(M)=l$ and $T_R(M)=r$ then $M$ corresponds to an $(n,n-1-r,l+1)$-Gog pentagon. This lets us rephrase the conjecture of Behrend in terms of Gog and Magog pentagons.
\begin{conj}[Behrend, \cite{FischerEnumerationOfAlternatingSignTrianglesUsingAConstantTermApproach}] Let $n\in\mathbb{N}$ and $1\leq l\leq n-2<r\leq 2n-3$ such that $l+r<2n-2$ and $r-l> n-3$. The number of $(n,r+2-n,r-l)$-Magog pentagons is the same as the number of $(n,r+2-n,l+1)$-Gog pentagons.
\end{conj}
\begin{exam}The set of nine $(4,2,2)$-Magog pentagons and $(4,2,3)$-Gog pentagons for $n=4$, $r=4$, $l=2$.
\begin{gather*}
    \begin{matrix}
        1&1
    \end{matrix}\quad \textrm{,}\quad 
    \begin{matrix}
        1&2
    \end{matrix}\quad \textrm{,}\quad 
    \begin{matrix}
        1&3
    \end{matrix}\quad \textrm{,}\quad 
    \begin{matrix}
        1&4
    \end{matrix}\quad \textrm{,}\quad 
    \begin{matrix}
        2&2
    \end{matrix}\quad \textrm{,}\quad 
    \begin{matrix}
        2&3
    \end{matrix}\quad \textrm{,}\quad \\
    \begin{matrix}
        2&4
    \end{matrix}\quad \textrm{,}\quad 
    \begin{matrix}
        3&3
    \end{matrix}\quad \textrm{,}\quad 
    \begin{matrix}
        3&4
    \end{matrix} \\
     \begin{matrix}
        &2\\
        1
    \end{matrix}\quad \textrm{,}\quad 
    \begin{matrix}
        &3\\
        1
    \end{matrix}\quad \textrm{,}\quad 
    \begin{matrix}
        &4\\
        1
    \end{matrix}\quad \textrm{,}\quad 
    \begin{matrix}
        &2\\
        2
    \end{matrix}\quad \textrm{,}\quad 
    \begin{matrix}
        &3\\
        2
    \end{matrix}\quad \textrm{,}\quad 
    \begin{matrix}
        &4\\
        2
    \end{matrix}\quad \textrm{,}\quad \\ 
    \begin{matrix}
        &3\\
        3
    \end{matrix}\quad \textrm{,}\quad 
    \begin{matrix}
        &4\\
        3
    \end{matrix}\quad \textrm{,}\quad 
    \begin{matrix}
        &4\\
        4
    \end{matrix}
\end{gather*}
\end{exam}
Note that the conjecture cannot be extended to more general parameters, i.e. we do not have that the set of $(n,k,l)$-Gog pentagons is equinumerous to the set of $(n,k,k-l-1+n)$-Magog pentagons for $1\leq l\leq n$ and $1\leq k\leq n$. For $n=3$, we see that for $k=2$ and $l=3$, there exists a unique $(3,2,3)$-Gog pentagon given by
$$\begin{matrix}
    1&2&3\\1&2\\1
\end{matrix}\quad ,$$
whereas there exist three $(3,2,1)$-Magog pentagons given by
$$\begin{matrix}
    1&1&1\\1&1\\1
\end{matrix}\quad , \quad
\begin{matrix}
    1&1&2\\1&1\\1
\end{matrix}\quad ,\quad
\begin{matrix}
    1&1&3\\1&1\\1
\end{matrix}\quad .$$

Also note that the other statistics mentioned by Behrend can also be included in the above discussion, giving a full reformulation of his conjecture.

There is another conjecture by Biane and Cheballah \cite{GogAndGogampentagons}, relating Gog pentagons to GOGAm pentagons, objects introduced in \cite{GogAndGogampentagons}. These GOGAm pentagons are obtained by applying the Schützenberger-involution to $(n,n)$-Magog trapezoids and then cutting off certain parts of the so obtained triangle. They conjectured that there are equally many Gog pentagons of a fixed shape as there are GOGAm pentagons of that shape. This is different in our conjectured correspondence as the Magog pentagons have a different shape than the Gog pentagons in general. Nevertheless combining these two, we can state another conjecture concerning our new objects.  
\begin{conj}
    Let $n\in\mathbb{N}$ and $1\leq l\leq n-2<r\leq 2n-3$ such that $l+r<2n-2$ and $r-l>n-3$. The number of $(n,r+2-n,r-l)$-Magog pentagons is the same as the number of $(0,n,r+2-n,l+1)$-GOGAm pentagons.
\end{conj}

The following observation regarding Theorem \ref{PfaffianCor} is also worth noting.
We know that in the unrestricted case the number of ASTs of order $n$ is given by a simple product formula as mentioned in the introduction. When considering $(n,l,r)$-ASPs this does not seem to generalise when considering arbitrary $r$ and $l$, so the Pfaffian obtained in Theorem \ref{PfaffianCor} does not seem to have a simple product factorization. See that the tables below, giving the number of $(n,l,r)$-ASPs, contain too large prime factors to yield a simple product formula. Checking the OEIS-database it seems that none of the sequences except for the case $l=0$ and $r=n-1$ as well as its symmetric case $l=n-2$ and $r=2n-3$ is enumerated by a known sequence. In the two symmetric cases mentioned we obtain the Catalan numbers. But this is easy to see, since $(n,1)$-Magog trapezoids are simply sequences $a_1\leq ...\leq a_n$ such that $a_i\leq i$, which are known to be enumerated by the Catalan numbers (see \cite[p. 50]{EnumerativeCombinatoricsVol1}).
\begin{table}[h!]
\centering
\begin{tabular}{||c||c|c|c|}
    \hline
    $n=4$&$l=0$&$l=1$&$l=2$\\
    \hline \hline
    $r=3$&$2\cdot 7$ &$3^2$ &$0$   \\
    \hline
    $r=4$&$5\cdot 7$ &$2^2\cdot 7$ &$3^2$ \\
    \hline
    $r=5$&$2\cdot 3\cdot 7$ &$5\cdot 73$ &$2\cdot 7$  \\
    \hline
    \end{tabular}
\caption{Numbers of $(4,l,r)$-ASPs}
\end{table}
\begin{table}[h!]
\centering
\begin{tabular}{||c||c|c|c|c|}
    \hline
    $n=5$&$l=0$&$l=1$&$l=2$&$l=3$\\
    \hline \hline
    $r=4$&$2\cdot 3\cdot 7$ &$2^2\cdot 7$ &$0$&$0$   \\
    \hline
    $r=5$&$3\cdot 73$ &$2^3\cdot 23$ &$2^3\cdot 3^2$ &$0$ \\
    \hline
    $r=6$&$3^2\cdot 43$ &$3\cdot 5\cdot 23$ &$2^3\cdot 23$ &$2^2\cdot 7$ \\
    \hline
    $r=7$&$3\cdot 11\cdot 13$ &$3^2\cdot 43$ &$3\cdot 73$ &$2\cdot 3\cdot 7$ \\
    \hline
    \end{tabular}
\caption{Numbers of $(5,l,r)$-ASPs}
\end{table}
\begin{table}[h!]
\centering
 \begin{tabular}{||c||c|c|c|c|c|}
    \hline
    $n=6$&$l=0$&$l=1$&$l=2$&$l=3$&$l=4$\\
    \hline \hline
    $r=5$&$2^2\cdot 3\cdot 11$ &$2\cdot 3^2\cdot 5$ &$0$&$0$&$0$  \\
    \hline
    $r=6$&$2\cdot 797$ &$5^3\cdot 11$ &$3\cdot 199$ &$0$ &$0$\\
    \hline
    $r=7$&$2\cdot 11\cdot 13\cdot 17$ &$5^2\cdot 179$ &$2\cdot 5^3\cdot 11$ &$3\cdot 199$ &$0$\\
    \hline
    $r=8$&$7^2\cdot 11\cdot 13$ &$2\cdot 11\cdot 13\cdot 23$ &$5^2\cdot 179$ &$5^6\cdot 11$& $2\cdot 3^2\cdot 5$\\
    \hline
    $r=9$&$2^2\cdot 11\cdot 13^2$ &$7^2\cdot 11\cdot 13$ &$2\cdot 11\cdot 13\cdot 17$ &$2\cdot 797$&$2^2\cdot 3\cdot 11$ \\
    \hline
    \end{tabular}
\caption{Numbers of $(6,l,r)$-ASPs}
\end{table}
\begin{table}[h!]
\centering
    \begin{tabular}{||c||c|c|c|c|c|c|}
    \hline
    $n=7$&$l=0$&$l=1$&$l=2$&$l=3$&$l=4$&$l=5$\\
    \hline \hline
    $r=6$&$3\cdot 11\cdot 13$ &$3^3\cdot 11$ &$0$&$0$&$0$&$0$  \\
    \hline
    $r=7$&$5\cdot 13\cdot 199$ &$11\cdot 1031$ &$3^2\cdot 11\cdot 53$ &$0$ &$0$&$0$\\
    \hline
    $r=8$&$5\cdot 11\cdot 13\cdot 107$ &$3\cdot 11\cdot 13\cdot 167$ &$2\cdot 11^2\cdot 197$ &$11\cdot 31\cdot 37$ &$0$&$0$\\
    \hline
    $r=9$&$2^3\cdot 3\cdot 13^2\cdot 41$ &$13\cdot 12253$ &$13^2\cdot 709$ &$2\cdot 11^2\cdot 197$& $3^2\cdot 11\cdot 53$&$0$\\
    \hline
    $r=10$&$2^5\cdot 3\cdot 13^3$ &$2^2\cdot 7\cdot 13^2\cdot 43$ &$13\cdot 12253$ &$3\cdot 11\cdot 13\cdot 167$&$11\cdot 1031$&$3^3\cdot 11$ \\
    \hline
    $r=11$&$2^2\cdot 13^2\cdot 17\cdot 19$ &$2^5\cdot 3\cdot 13^3$ &$2^3\cdot 3\cdot 13^2\cdot 41$ &$5\cdot 11\cdot 13\cdot 107$&$5\cdot 13\cdot 199$&$3\cdot 11\cdot 13$ \\
    \hline
    \end{tabular}
\caption{Numbers of $(7,l,r)$-ASPs}
\end{table}

\newpage
\bibliographystyle{alpha}
\bibliography{main}

\end{document}